\documentclass[a4paper,12pt]{amsart}

\usepackage{amssymb}
\usepackage{amsmath}
\usepackage{amsthm}

\usepackage{amsfonts}
\usepackage{mathrsfs}
\usepackage[colorlinks]{hyperref}
\usepackage{graphicx}
\usepackage{algorithm}
\usepackage{listings}
\usepackage{algorithm,algorithmic,enumerate}
\usepackage{amsmath,tikz}
\usetikzlibrary{matrix}

\numberwithin{equation}{section}
\theoremstyle{plain}
\newtheorem{thm}{Theorem}[section]

\newtheorem{cor}[thm]{Corollary}
\newtheorem{lemma}[thm]{Lemma}

\newtheorem{claim}[thm]{Assumption}

\newtheorem{problem}[thm]{Problem}

\theoremstyle{definition}
\newtheorem{defi}[thm]{Definition}
\newtheorem{rem}[thm]{Remark}

\begin{document}

\title[Inverse problem of determining leading coefficient]{Inverse problem of determining time-dependent leading coefficient in the time-fractional heat equation}

\author[D. Serikbaev]{Daurenbek Serikbaev}
\address{
  Daurenbek Serikbaev:
   \endgraf
  Department of Mathematics: Analysis, Logic and Discrete Mathematics
  \endgraf
  Ghent University, Belgium
  \endgraf
  and
  \endgraf   
   Al--Farabi Kazakh National University
  \endgraf
  Almaty, Kazakhstan
  \endgraf
  and
  \endgraf 
   Institute of Mathematics and Mathematical Modeling
  \endgraf
  Almaty, Kazakhstan
  \endgraf
  {\it E-mail address} {\rm daurenbek.serikbaev@ugent.be}
  }
\author[M. Ruzhansky]{Michael Ruzhansky}
\address{
  Michael Ruzhansky:
  \endgraf
Department of Mathematics: Analysis,
Logic and Discrete Mathematics
  \endgraf
Ghent University, Belgium
  \endgraf
 and
  \endgraf
 School of Mathematical Sciences
 \endgraf
Queen Mary University of London
\endgraf
United Kingdom
\endgraf
  {\it E-mail address} {\rm michael.ruzhansky@ugent.be}
 }

\author[N. Tokmagambetov ]{Niyaz Tokmagambetov }
\address{
  Niyaz Tokmagambetov:
  \endgraf 
  Centre de Recerca Matem\'atica
  \endgraf
  Edifici C, Campus Bellaterra, 08193 Bellaterra (Barcelona), Spain
  \endgraf
  and
  \endgraf   
  Institute of Mathematics and Mathematical Modeling
  \endgraf
  125 Pushkin str., 050010 Almaty, Kazakhstan
  \endgraf  
  {\it E-mail address:} {\rm tokmagambetov@crm.cat; tokmagambetov@math.kz}
  }

\date{\today}

\thanks{This research was funded by the Science Committee of the Ministry of Education and Science of the Republic of Kazakhstan (Grant No. AP14872042), by the FWO Odysseus 1 grant G.0H94.18N: Analysis and Partial Differential Equations, and by the Methusalem programme of the Ghent University Special Research Fund (BOF) (Grant number 01M01021). MR is also supported by EPSRC grant EP/R003025/2. NT is also supported by the Beatriu de Pin\'os programme and by AGAUR (Generalitat de Catalunya) grant 2021 SGR 00087.}


\maketitle

\begin{abstract}
In this paper, we investigate direct and inverse problems
for the time-fractional heat equation with a time-dependent leading coefficient for positive operators. 
First, we consider the direct problem, and the unique existence of the generalized solution is established. We also deduce some regularity results. Here, our proofs are based on the eigenfunction expansion method. Second, we study the inverse problem of determining the leading coefficient, and the well-posedness of this inverse problem is proved. 
\end{abstract}
\tableofcontents

\tableofcontents

\section{Introduction}
Let $T$ be a positive constant and let $\mathcal{H}$ be a separable Hilbert space. In this paper, we consider the following time-fractional heat equation
\begin{equation}\label{EQ:Frac Pseudo}
    \mathcal{D}_{t}^\alpha v(t)+\sigma(t)\mathcal{M}v(t)=f(t)\; \text{in} \;\mathcal{H},
\end{equation}
for $0<t\leq T.$ Here $\mathcal{D}_t^\alpha$ is the Caputo derivative of the order $\alpha\in(0,1)$ defined by
$$\mathcal{D}_t^\alpha v(t)=\frac{1}{\Gamma(1-\alpha)}\int_0^t (t-\tau)^{-\alpha}\frac{d}{d\tau} v(\tau)d\tau,$$
with Gamma function $\Gamma(\cdot),$ and 
$\mathcal{M}$ is the positive self-adjoint operator with the discrete spectrum $\{\mu_\xi\}_{\xi\in\mathcal{I}}$ such that $\mu_\xi\to\infty$ as $|\xi|\to\infty$ and the system of eigenfunctions $\{\omega_\xi\}_{\xi\in\mathcal{I}}$ forming an orthonormal basis in the space $\mathcal{H},$ where $\mathcal{I}$ is a countable set.

In this paper, we consider direct and inverse problems.

First, we start by stating the direct problem, the properties of which we are going to use in studying the inverse problem of determining the coefficient $\sigma(t).$
\begin{problem}[Direct problem]
\label{P:D} 
Given $\sigma(t)$ and $f(t),$ find a function $v(t)$ such that $v:[0,T]\to\mathcal{H}$ satisfies the equation \eqref{EQ:Frac Pseudo} and the initial condition
\begin{equation}\label{CON:IN}
    v(0)=h\;\text{in} \;\mathcal{H}.
\end{equation}
\end{problem}
For this direct problem, we prove the uniquely existence of the generalised solution and derive some regularity results in Section \ref{S:D}. In the main part of the paper, based on Problem \ref{P:D}, we consider the following inverse problem of finding the coefficient $\sigma(t)$ in the equation \eqref{EQ:Frac Pseudo} for the operator $\mathcal{M}$ with the positive discrete spectrum such that $\inf_{\xi\in\mathcal{I}}\mu_\xi>0,$ from the given additional information.

\begin{problem}[Inverse Problem]\label{P:DIP}
Given $f(t)$ and $h,$ find a pair of functions $(\sigma,v)$ satisfying the problem \eqref{EQ:Frac Pseudo}--\eqref{CON:IN} and the additional condition
\begin{equation}\label{CON:ADD}
    F[v(t)]=E(t),\; t\in[0,T].
\end{equation}
\end{problem}

In \eqref{CON:ADD}, $F$ is a \textbf{linear bounded functional}
$$F:\mathcal{H}^{1+\gamma}\rightarrow \mathbb{R}.$$ 
Here $\mathcal{H}^{1+\gamma}=\{v\in\mathcal{H}:\,\mathcal{M}^{1+\gamma}v\in\mathcal{H}\}$ and $\gamma\geq 0,$ and $F$ satisfies the following assumption:
\begin{equation}\label{gamma F}
    \left\{\frac{F[\omega_\xi]}{\mu_\xi^\gamma}\right\}\in l^2(\mathcal{I}).
\end{equation}
For example, let $\mathcal{H}$ be $L^2(0,1)$ and $\mathcal{M}v=-v_{xx},\; x\in (0,1),$ with homogeneous Dirichlet boundary condition. Then the operator has the eigensystem $\{(\pi k)^2, \sqrt{2}\sin{k\pi x}\}_{k\in\mathbb{N}},$ with $\mathcal{I}=\mathbb{N}.$ 
Let $F[v(t,\cdot)]:=\int_{0}^1 v(t,x)dx.$ 
Since 
\begin{equation*}F[\omega_k]=\int_0^1 \sqrt{2}\sin{k\pi x}dx=\frac{\sqrt{2}(1+(-1)^{k+1})}{k\pi}=\begin{cases}
    &\frac{2\sqrt{2}}{k\pi},\;\text{if},\;k=2n-1\;(n\in\mathbb{N});\\
    &0, \;\text{if},\;k=2n\;(n\in\mathbb{N}),
\end{cases}\end{equation*} 
we have 
$$\sum_{k\in\mathbb{N}} |F[\omega_k]|^2<\infty.$$ From this we see that $\gamma$ in \eqref{gamma F} can be taken to be $0.$ 

Without loss of generality, we can assume
\begin{equation}\label{Positiveness F}
   \text{that}\;F\not\equiv 0\; \text{and that}\; F[\omega_\xi]\geq 0, \text{for all}\;\xi\in\mathcal{I}.
\end{equation}

If, on the contrary, we find that $F[\omega_\zeta]<0$ for a certain $\zeta\in\mathcal{I}$, we can address this by considering $-\omega_\zeta$ instead. By doing so, we ensure that the properties we require are still satisfied. Specifically, $-\omega_\zeta$ remains an eigenfunction of $\mathcal{M}$ associated with the eigenvalue $\mu_\zeta$, it contributes to an orthonormal basis of $\mathcal{H}$, and we have $F[-\omega_\zeta]\geq 0$.

For this inverse problem, we obtain the following results:
\begin{itemize}
    \item The existence of the solution;
    \end{itemize}
To prove the existence of the solution of Problem \ref{P:DIP}, we reduce this inverse problem to the operator equation for the diffusion coefficient $\sigma(t).$ Then we prove the existence of a solution of this operator equation using Shauder's fixed-point theorem. 

When $\alpha=1,$ Shauder's fixed point theorem is used in many works, here we only note a few of them \cite{I93, KI12, IO16, HLI16}. To the best of our knowledge, in this paper we use Shauder's fixed point theorem to prove the existence of the solution to the obtained operator equation for the first time. That is why we will show the existence of the solution of the obtained operator equation analogously to the proof of the above-mentioned papers. Using Shauder's fixed point theorem to prove the existence result is valuable since it can be extended to other inverse problems for time-fractional differential equations.
    \begin{itemize}
        \item The continuous dependence on the data;
    \end{itemize} 
To show the continuous dependence of the solution of Problem \ref{P:DIP}, we use Gronwall type inequality \cite[Lemma 7.1.1]{H81}.
\begin{itemize}
        \item The uniqueness of the solution;
    \end{itemize}
We prove the uniqueness of the solution to Problem 1.2 by relying on the results of the continuous dependence part.

In recent years, there has been significant interest in the direct and inverse problems of time-fractional diffusion equations. Several authors have contributed to this field through their classical papers, which have explored various aspects of this problem. 

Here we will mention only the closest scientific works to our research which were done for particular cases of the operator $\mathcal{M}$. 

The problem \eqref{EQ:Frac Pseudo}--\eqref{CON:IN} when $\sigma(t)=1$ was studied in \cite{SY11}. Under suitable assumptions on given functions, the authors proved the existence and uniqueness of weak solutions to the problem \eqref{EQ:Frac Pseudo}--\eqref{CON:IN}. Moreover, the authors proved other stability and uniqueness results for some related inverse problems. This was one of the first mathematical works concerning fractional inverse problems. The maximum principle for the time-fractional diffusion equation was established by Luchko in \cite{L09}. 

For the one dimensional case, the model \eqref{EQ:Frac Pseudo}--\eqref{CON:IN} was considered in \cite{Zh16} with a homogeneous right-hand side and trivial initial condition. Z. Zhang established the existence and
uniqueness of the weak solution and deduced some regularity results. Moreover, Z. Zhang considered the inverse problem of recovering $\sigma(t)$ in \cite{Zh16}. For this inverse problem, he showed the uniqueness of $\sigma(t).$

Z. Zhang also analyzed the direct problem \eqref{EQ:Frac Pseudo}--\eqref{CON:IN} on $\mathbb{R}^n$ in \cite{Zh17}. In particular, he established the existence, uniqueness, and some regularity properties. In the same paper, Z. Zhang considered an inverse problem for recovering $\sigma(t).$ For this inverse problem he used the single point flux data
$$\sigma(t)\nabla{v}(x^*,t)\cdot \vec{n}=E(t),\;x^*\in \partial\Omega,$$
to identify the coefficient $\sigma(t).$ 
We note that this additional information does not intersect with the additional data \eqref{CON:ADD}.
In the context of recovering $\sigma(t),$ Z. Zhang introduced an appropriate operator $P$ and established its monotonicity. This fundamental property of the operator $P$ implies both the uniqueness and existence of the solution $\sigma(t)$. In other words, the author established the uniqueness and existence of the solution of the considered inverse problem by using the monotonicity of the operator $P$. This method is also applicable for Problem \ref{P:DIP}. The disadvantage of this method is that it works only for the inverse problem of identification of a time-dependent leading coefficient, but not for other types of inverse problems for time-fractional differential equations. For example, this method does not work for inverse source problems. 

Furthermore, Lopushanskyi and Lopushanska \cite{LL14} examined the same model as investigated in \cite{Zh16} for the case $\alpha\in(0, 2)$. They employed the Green function to obtain a representation of the solution $v$. Additionally, they introduced an operator for $\sigma(t)$ that guarantees the existence and uniqueness of the pair $(\sigma,v).$ 

This paper follows the following structure. The Section \ref{S:prelim} consists of preliminary results about Mittag-Leffler functions, fractional calculus, the generalized mean value theorem, the Gronwall type lemma, and definitions of generalized Sobolev spaces. In Section \ref{S:D}, we consider direct Problem \ref{P:D}, and prove the existence and uniqueness of its generalized solution. Also, we show regularity of this solution. In Section \ref{S:Inv}, we consider inverse Problem \ref{P:DIP}, and prove its well-posedness. In Sections \ref{S:L} and \ref{S:F}, we give examples of the operator $\mathcal{M}$ and the functional $F,$ respectively. 
In Section \ref{S:gamma}, for the particular cases of $\mathcal{M}$ and ${F},$ we show how to find the value of $\gamma$ in \eqref{gamma F}.
In the last section, we include as an Appendix some classical theorems which are used in this paper.

\section{Preliminary materials}\label{S:prelim}

\subsection{Generalized Mittag-Leffler function} In this subsection, we will recall the definition of the generalized Mittag-Leffler function and give its necessary properties which we will use in our investigation. The generalized Mittag-Leffler function is
$$E_{\alpha,\beta}(z)=\sum_{k=0}^\infty \frac{z^k}{\Gamma(k\alpha+\beta)},\;z\in\mathbb{C}.$$

\begin{lemma}\label{L:converge E}\cite[Theorem 1.6]{P99} 
Suppose that $\alpha\in(0,2)$, $\beta$ is an arbitrary real number and, $\pi\alpha/2<\mu<\min\{\pi,\pi\alpha\}$. Then
there exists a positive constant $C'$ such that the inequality holds:
\begin{equation*}\label{EST: Mittag}
|E_{\alpha,\beta}(z)|\leq\frac{C'}{1+|z|},
\end{equation*}
for all $\mu\leq|arg(z)|\leq\pi$ and $|z|\geq 0.$
\end{lemma}
Note that in \cite{Sim14} the following estimate for the Mittag-Leffler function is proved, when $0<\alpha<1$ (not true for $\alpha\geq 1$)
\begin{equation*}
\frac{1}{1+\Gamma(1-\alpha)z}\leq E_{\alpha,1}(-z)\leq\frac{1}{1+\Gamma(1+\alpha)^{-1}z},\;z>0,
\end{equation*}
Thus, it follows that
\begin{equation}\label{EST: Mittag 0}
0<E_{\alpha,1}(-z)<1,\;z>0.
\end{equation}
\begin{lemma}\label{L:Mon E}
If $0<\alpha\leq 1$ and $\beta\geq \alpha.$ Then the generalized Mittag-Leffler
function $E_{\alpha,\beta}(-z),\;z\geq 0,$ is completely monotonic, that is,
$$(-1)^n\frac{d^n}{dx^n}E_{\alpha,\beta}(-z)\geq 0,\;\text{for}\; z\geq 0\; \text{and}\; n=0,1,2,\dots$$
\end{lemma}
\begin{proof}
The proof for this result can be found in the references \cite{P48, Sch96, MS97, MS01}.
\end{proof}
\begin{lemma}\label{L:E1E2}
For $0<\alpha\leq 1$ and $\beta\geq \alpha,$ the Mittag-Leffler type function $E_{\alpha,\beta}(-\mu_\xi t^\alpha)$ satisfies
$$0\leq E_{\alpha,\beta}(-\mu_\xi t^\alpha)\leq E_{\alpha,\beta}(-\inf_{\xi\in\mathcal{I}}\mu_\xi t^\alpha)\leq \frac{1}{\Gamma(\beta)},\; t\geq 0.$$
\end{lemma}
\begin{proof}
    In view of Lemma \ref{L:Mon E}, we have
    $$E_{\alpha,\beta}(-z)\geq 0,\;\forall z\geq 0,$$
and
\begin{equation*}
    \begin{split}
       \frac{d}{dz}E_{\alpha,\beta}(-z)\leq 0,\;\forall z\geq 0.
    \end{split}
\end{equation*}
These estimates mean that the function $E_{\alpha,\beta}(-z)$ is positive and non-increasing on the interval $[0,\infty)$, i.e. it satisfies the following inequality for all $x,\,y$ such that $x\geq y\geq 0$
$$
       0\leq E_{\alpha,\beta}(-x)\leq E_{\alpha,\beta}(-y)\leq E_{\alpha,\beta}(0)=\frac{1}{\Gamma(\beta)}.
$$
These yield the desired results and complete the proof. 
\end{proof}
We now present a particular case of the lemma from \cite[Lemma 3.2]{SY11}. The proof mainly follows the arguments of the proof in \cite[Lemma 3.2]{SY11}, hence, we omit it. \begin{lemma}\label{L:Der E}
If $\mu>0$ and $\alpha>0$, then
$$\frac{d}{dt}E_{\alpha,1}(-\mu t^\alpha)=-\mu t^{\alpha-1}E_{\alpha,\alpha}(-\mu t^\alpha), \forall{t}\in (0,\infty).$$
\end{lemma}

\subsection{Fractional calculus} In this subsection, we compile several results from fractional calculus.
\begin{lemma}\cite[Lemma 2.5]{Zh17}\label{L:max+min}
Let $\alpha\in(0,1)$ and $v=v(t)\in C[0,T]$ with $\mathcal{D}_t^\alpha v\in C[0,T].$ 
Then the following statements hold:

$(a)$ If $t_0\in (0,T]$ and $v(t_0)=\max_{t\in[0,T]}v(t),$ then we have $\mathcal{D}_t^\alpha v\geq 0.$ 

$(b)$ If $t_0\in (0,T]$ and $v(t_0)=\min_{t\in[0,T]}v(t),$ then we have $\mathcal{D}_t^\alpha v\leq 0.$
\end{lemma}

\begin{lemma}\label{I0D}\cite[Lemma 2.6]{Zh17}
The Riemann-Liouville integral $I^\alpha_t$ of order $\alpha>0$ is defined as
$$I_t^\alpha v=\frac{1}{\Gamma(\alpha)}\int_0^t (t-\tau)^{\alpha-1}v(\tau)d\tau.$$
For $0<\alpha<1,$ $v(t)$ and $\mathcal{D}_t^\alpha v(t)$ are both continuous on the interval $[0,T],$ then we have
$$(\mathcal{D}_t^\alpha \circ I_t^\alpha v)(t)=v(t),\;(I_t^\alpha\circ \mathcal{D}_t^\alpha v)(t)=v(t)-v(0),\;t\in [0,T].$$
\end{lemma}

\subsection{Generalized mean value theorem} In this part, we present the generalized mean value theorem. 
\begin{thm}\label{Cor}\cite[Theorem 1]{OSh07} Let $0<\alpha\leq 1,$ $a<b$ and $g\in C[a,b]$ be such that ${}_a\mathcal{D}^\alpha_t g\in C[a,b].$ Then, there exists some $t\in (a,b)$ such that
$$
g(b)-g(a)=\frac{1}{\Gamma(\alpha+1)}{}_a\mathcal{D}^\alpha_t g(t)\cdot (b-a)^\alpha.
$$
\end{thm}
Here ${}_a\mathcal{D}^\alpha_t $ is the Caputo derivative defined by
$${}_a\mathcal{D}^\alpha_t g=\frac{1}{\Gamma(1-\alpha)}\int_a^t (t-\tau)^{-\alpha}\frac{d}{d\tau} g(\tau)d\tau .$$
Theorem \ref{Cor} plays a crucial role in establishing the existence of a solution to our inverse problems.

\subsection{Gronwall type lemma}
We present, an inequality of Gronwall type with weakly singular kernel $(t-\tau)^{\alpha-1}$ (see \cite[Lemma 7.1.1]{H81}).
\begin{lemma}\label{L:Gronwall}
    Suppose $c\geq 0,\,0<\alpha<1,$ and $z(t)$ is a non-negative function locally integrable on $[0,b)$ (for some $b\leq\infty$), and suppose $y(t)$ is non-negative and locally integrable on $[0,b)$ with
    $$y(t)\leq z(t)+c\int_0^t(t-\tau)^{\alpha-1}y(\tau)d\tau,\;\forall t\in[0,b).$$
    Then 
    $$y(t)\leq z(t)+c\Gamma(\alpha)\int_0^t \frac{d}{d\tau}E_{\alpha,1}(c\Gamma(\alpha)(t-\tau)^\alpha)z(\tau)d\tau, \;\forall t\in[0,b).$$
    If $z(t)\equiv z$ is constant, then $$y(t)\leq zE_{\alpha,1}(c\Gamma(\alpha)t^\alpha),\;\forall t\in[0,b).$$
\end{lemma}

\subsection{Generalized Sobolev spaces} In this section, we fix a few definitions concerning the generalized Sobolev space over $\mathcal{H}.$

\begin{defi}\label{d:H rho}
Let $\rho\in\mathbb{R}$. We denote the Sobolev space by $\mathcal{H}^{\rho}:=\{v\in\mathcal{H}:\; \mathcal{M}^\rho v\in \mathcal{H}\}$ with the norm
$$\|v\|_{\mathcal{H}^\rho}=\left(\sum_{\xi\in\mathcal{I}}|(1+\mu_\xi)^\rho(v,\omega_\xi)_{\mathcal{H}}|^2\right)^\frac{1}{2}.$$
\end{defi}
\begin{defi}
Let $\rho\in\mathbb{R}$. For $0<\alpha<1$ we denote by ${X}^\alpha([0,T];\mathcal{H}^{\rho})$ the space of all continuous functions $g:[0,T]\rightarrow \mathcal{H}^\rho$ with also continuous $\mathcal{D}^\alpha_t g:[0,T]\rightarrow \mathcal{H}^\rho$, such that
$$
\|g\|_{{X}^\alpha([0,T];\mathcal{H}^{\rho})}:=\|g\|_{C([0,T];\mathcal{H}^{\rho})}+\|\mathcal{D}_t^\alpha g\|_{C([0,T];\mathcal{H}^{\rho})}<\infty.
$$
The space ${X}^\alpha([0,T];\mathcal{H}^{\rho})$ equipped with the norm above is a Banach space.
\end{defi}
\begin{defi}
For $\theta\in(0,1)$ we set
$$C^\theta([0,T];\mathcal{H})=\left\{f\in C([0,T];\mathcal{H}):\;\sup_{0\leq t<s\leq T}\frac{\|f(t)-f(s)\|_\mathcal{H}}{|t-s|^\theta}<\infty\right\}$$
and
$$\|f\|_{C^\theta([0,T];\mathcal{H})}=\|f\|_{C([0,T];\mathcal{H})}+\sup_{0\leq t<s\leq T}\frac{\|f(t)-f(s)\|_\mathcal{H}}{|t-s|^\theta}.$$
\end{defi}

\section{Existence and Uniqueness of the solutions of Problem \ref{P:D}}\label{S:D}

In this section, we consider Problem \ref{P:D} and prove the existence and uniqueness of the generalized solution and deduce some regularity results. 

During this section, let us suppose that $a,\,h$ and $f$ satisfy the following assumptions:
\begin{claim}\label{A:D}

$(I)$ $\sigma\in \mathcal{C}^+[0,T]:=\{\sigma\in C[0,T]:\;\sigma(t)\geq m_{\sigma}>0,\,t\in[0,T]\};$

$(II)$ $f\in C([0,T]; \mathcal{H});$

$(III)$ $h\in \mathcal{H}^1.$
\end{claim}

A generalised solution of Problem \ref{P:D} is the function $v(t;\sigma)$ such that $v\in C([0,T];\mathcal{H}^1)$ together with $\mathcal{D}_t^\alpha v\in C([0,T];\mathcal{H}).$ 
The notation $v(t;\sigma)$ is used to indicate the dependence of the generalized solution $v$ on the diffusion coefficient $\sigma(t).$

For Problem \ref{P:D}, we now present the main theorem of this section.

\begin{thm}\label{thm:MainThmDP}
Let Assumption \ref{A:D} be satisfied. Then there exists the unique generalized solution, denoted as $v(t;\sigma)$, to \ref{P:D}, which can be represented as \begin{equation}\label{EXPANTION u}
    v(t;\sigma)=\sum_{\xi\in\mathcal{I}}v_\xi(t;\sigma)\omega_\xi,\; t\in [0,T].
\end{equation}
Furthermore, one has the following regularity estimates:

$i)$ If $f\in C^\alpha([0,T];\mathcal{H}),$ then
$$\|v\|_{C([0,T];\mathcal{H}^1)}+\|\mathcal{D}_t^\alpha v\|_{C([0,T];\mathcal{H})}\leq C(\|h\|_{\mathcal{H}^1}+(T^\alpha+1)\|f\|_{C^\alpha([0,T];\mathcal{H})}).$$

$ii)$ If $f\in C([0,T];\mathcal{H}^\frac{1}{2}),$ then
$$\|v\|_{C([0,T];\mathcal{H}^1)}+\|\mathcal{D}_t^\alpha v\|_{C([0,T];\mathcal{H})}\leq C\big(\|h\|_{\mathcal{H}^1}+(T^{\alpha}+T^\frac{\alpha}{2}+1)\|f\|_{C([0,T];\mathcal{H}^\frac{1}{2})}\big).$$

$iii)$ If $f\in C([0,T];\mathcal{H}^\frac{1}{2})$ and $\inf_{\xi\in\mathcal{I}}\mu_\xi>0,$ then
$$\|v\|_{C([0,T];\mathcal{H}^1)}+\|\mathcal{D}_t^\alpha v\|_{C([0,T];\mathcal{H})}\leq C\big(\|h\|_{\mathcal{H}^1}+\left(\frac{1}{\sqrt{\inf_{\xi\in\mathcal{I}}\mu_\xi}}+1\right)\|f\|_{C([0,T];\mathcal{H}^\frac{1}{2})}\big),$$
where $C>0$ does not depend on $T.$
\end{thm}

\subsection{Proof of Theorem \ref{thm:MainThmDP}} First, we prove unique existence of the generalized solution $v(t;\sigma)$ of Problem \ref{P:D} in subsection \ref{S:unieque existence}. Second, we show regularity  properties of $v(t;\sigma)$ in subsection \ref{S:regularity}.

\subsubsection{Existence and Uniqueness of $v(t;\sigma)$}\label{S:unieque existence} We now present the theorem about uniqueness and existence of the solution of Problem \ref{P:D}.
\begin{thm}\label{Th:DEU} 
Let Assumption \ref{A:D} hold. Then the following statements hold true:
    
    1) For each $\xi\in\mathcal{I}$ there is a unique solution $v_\xi(t;\sigma)$ of the following problem
    \begin{equation}\label{FODE}
    \mathcal{D}_t^\alpha v_\xi(t;\sigma)+\mu_\xi \sigma(t)v_\xi(t;\sigma)=f_\xi(t),\;v_\xi(0;\sigma)=h_\xi,\;\xi\in\mathcal{I},
\end{equation} where $h_\xi=(h,\omega_\xi),\;f_\xi(t)=(f(t),\omega_\xi),\;\xi\in\mathcal{I}.$ This solution belongs to $X^\alpha[0,T];$

    2) Problem \ref{P:D} has a unique generalised solution $v(t;\sigma)$ with the spectral representation \eqref{EXPANTION u}.
\end{thm}
\begin{proof}
Let us first prove the existence and uniqueness of the solution $v_\xi(\cdot;\sigma)\in C[0, T]$ of \eqref{FODE} for each $\xi\in\mathcal{I}.$ To this end for each $\xi\in\mathcal{I}$ we rewrite the Cauchy-type problem \eqref{FODE} in the following form:
\begin{equation}\label{FODE 1}
    \begin{cases}
    &\mathcal{D}_t^\alpha v_\xi(t;\sigma)+\mu_\xi M_{\sigma} v_\xi(t;\sigma)=f_\xi(t)+\mu_\xi (M_{\sigma}-\sigma(t))v_\xi(t;\sigma),\;t\in(0,T],\\
    &v_\xi(0;\sigma)=h_\xi,
\end{cases}\end{equation}
where $M_{\sigma}$ is a constant such that
\begin{equation}\label{q<Q}
   0<m_{\sigma}\leq \sigma(t)<M_{\sigma}\;\text{on}\;[0,T].
\end{equation}
Assumption \ref{A:D} $(I)$ guaranties the existence of $M_{\sigma}>m_{\sigma}.$

According to \cite{LG99}, the solution of equation \eqref{FODE 1} can be represented in the form
\begin{equation}\label{Volt int}
    \begin{split}
        v_\xi(t;\sigma)&=h_\xi E_{\alpha,1}(-\mu_\xi M_{\sigma} t^\alpha)\\
        &+\int_0^t \left[f_\xi(s)+\mu_\xi(M_{\sigma}-\sigma(s))v_\xi(s;\sigma)\right](t-s)^{\alpha-1}E_{\alpha,\alpha}(-\mu_\xi M_{\sigma} (t-s)^\alpha)ds,
    \end{split}
\end{equation}
for each $\xi\in\mathcal{I}.$ Since the solution $v_\xi(t;\sigma)$ of \eqref{FODE 1} is represented as \eqref{Volt int}, to prove the existence and uniqueness of the solution $v_\xi(\cdot;\sigma)\in C[0,T]$ of \eqref{FODE} it is sufficient to prove the existence of a unique solution
$v_\xi(\cdot;\sigma)\in C[0,T]$ of the equation \eqref{Volt int}. For this, we use the Banach fixed point Theorem \ref{Th:BF} for the space $C[0, T].$ We rewrite the integral equation \eqref{Volt int} in the form $v_\xi(t;\sigma)=(Av_\xi)(t;\sigma),$ where 
\begin{equation}\label{Volt OP}
    \begin{split}
        (Av_\xi)(t;\sigma)&=h_\xi E_{\alpha,1}(-\mu_\xi M_{\sigma} t^\alpha)\\
        &+\int_0^t \left[f_\xi(s)+\mu_\xi(M_{\sigma}-\sigma(s))v_\xi(s;\sigma)\right](t-s)^{\alpha-1}E_{\alpha,\alpha}(-\mu_\xi M_{\sigma} (t-s)^\alpha)ds.
    \end{split}
\end{equation}
 To apply the Banach fixed point Theorem \ref{Th:BF}, we have to prove the following: 
 $i)$ if $v_\xi\in C[0,T],$ then $Av_\xi\in C[0,T];$ $ii)$ for any $u_\xi,\,w_\xi\in C[0,T],$ the following estimate holds:
 \begin{equation}\label{Tu-Tv}
     \|Au_\xi-Aw_\xi\|_{C[0,T]}\leq \beta \|u_\xi-w_\xi\|_{C[0,T]},\;\beta\in(0,1).
 \end{equation}
We first prove $i).$ To this end we take $v_\xi(\cdot;\sigma)\in C[0,T]$ and we denote
\begin{equation}\label{T11}
    (A_1 f_\xi)(t)=\int_0^t f_\xi(s)(t-s)^{\alpha-1}E_{\alpha,\alpha}(-\mu_\xi M_{\sigma} (t-s)^\alpha)ds
\end{equation}
and 
\begin{equation}\label{T22}
    (A_2 v_\xi)(t;\sigma)=\int_0^t \mu_\xi(M_{\sigma}-\sigma(s))v_\xi(s;\sigma)(t-s)^{\alpha-1}E_{\alpha,\alpha}(-\mu_\xi M_{\sigma} (t-s)^\alpha)ds.
\end{equation}

Let us formulate the following lemma. 
\begin{lemma}\label{L:con f} Let Assumption \ref{A:D} $(II)$ hold true. Then for each $\xi\in \mathcal{I}$ the coefficient function $f_\xi(t)$ belongs to $C[0,T].$  
\end{lemma}
\begin{proof} In view of Assumption \ref{A:D} $(II)$ for any $\varepsilon>0$ exists $\delta=\delta(\varepsilon)$ such that $|t_1-t_2|<\delta$ $\forall t_1,t_2\in[0,T]:$
    \begin{equation*}\begin{split}
      \|f(t_1)-f(t_2)\|_\mathcal{H} <\varepsilon.  
    \end{split}\end{equation*}
This with the Bessel inequality gives us
\begin{equation*}\begin{split}
    \varepsilon&>\|f(t_1)-f(t_2)\|_\mathcal{H}\geq \left(\sum_{\xi\in\mathcal{I}}|f_\xi(t_1)-f_\xi(t_2)|^2\right)^\frac{1}{2}\\
    &\geq |f_\xi(t_1)-f_\xi(t_2)|.
\end{split}
\end{equation*}
This implies $f_\xi\in C[0,T]$ for each $\xi\in\mathcal{I}.$\end{proof} 
In view of Lemma \ref{L:Mon E}, the function $E_{\alpha,\alpha}(-z),\;z\geq 0$ possesses of derivatives $\frac{d^n}{dz^n}E_{\alpha,\alpha}(-z)$ for all $n\in\mathbb{N}.$ This implies that $E_{\alpha,\alpha}(-\mu_\xi M_{\sigma} t^\alpha)$ belongs to $C[0,T].$ Hence, the function $t^{\alpha-1}E_{\alpha,\alpha}(-\mu_\xi M_{\sigma} t^\alpha)$ is continuous in $t\in(0,T].$
This together with Lemma \ref{L:con f} tell us that $(A_1 f_\xi)(t)$ which is given by \eqref{T11} belongs to $C[0, T].$ At the beginning of proof $i)$ we suppose that $v_\xi(\cdot;\sigma)\in C[0,T]$ this in view of Lemma \ref{L:Der E} gives us that $(A_2 v_\xi)(t;\sigma)$ which is given by \eqref{T22} belongs to $C[0,T].$ Since the first term of \eqref{Volt OP} also belongs to $C[0,T],$ we have $Av_\xi\in C[0,T].$ 

Now we proceed to prove the estimate \eqref{Tu-Tv}. In view of \eqref{q<Q} and using \eqref{Volt OP}, we have
\begin{equation*}
    \begin{split}
    &|Au_\xi(t;\sigma)-Aw_\xi(t;\sigma)|\leq\\
    &\int_0^t |M_{\sigma}-\sigma(s)| |u_\xi(s;\sigma)-w_\xi(t;\sigma)| \mu_\xi(t-s)^{\alpha-1}E_{\alpha,\alpha}(-\mu_\xi M_{\sigma} (t-s)^\alpha)ds\\
    &\leq (M_{\sigma}-m_{\sigma})\|u_\xi-w_\xi\|_{C[0,T]}\int_0^t\mu_\xi(t-s)^{\alpha-1}E_{\alpha,\alpha}(-\mu_\xi M_{\sigma} (t-s)^\alpha)ds.
    \end{split}
\end{equation*}
Lemma \ref{L:Der E} gives
\begin{equation*}
   \int_0^t\mu_\xi(t-s)^{\alpha-1}E_{\alpha,\alpha}(-\mu_\xi M_{\sigma} (t-s)^\alpha)ds=\frac{1}{M_{\sigma}}\left(1-E_{\alpha,1}(-\mu_\xi M_{\sigma} t^\alpha)\right)\leq \frac{1}{M_{\sigma}}. 
\end{equation*}
Hence,
$$\|Au_\xi-Aw_\xi\|_{C[0,T]}\leq \frac{M_{\sigma}-m_{\sigma}}{M_{\sigma}}\|u_\xi-w_\xi\|_{C[0,T]}.$$
Here $\beta=\frac{M_{\sigma}-m_{\sigma}}{M_{\sigma}}<1.$ Hence by Banach fixed point Theorem \ref{Th:BF} there exists a unique solution $v_\xi=v_\xi^{*}\in C[0,T]$ for each $\xi\in\mathcal{I}$ to the equation \eqref{Volt int}.

To complete the proof of Theorem \ref{Th:DEU} $1)$ we must show that such a unique solution $v_\xi\in C[0,T]$ belongs to the space $X^\alpha[0,T].$ For this it is sufficient to prove that $\mathcal{D}_t^\alpha v_\xi\in C[0,T].$ 

Applying $\mathcal{D}_t^\alpha$ to \eqref{Volt int}, we have
\begin{equation}\label{Dau 0}
    \begin{aligned}
        &\mathcal{D}_t^\alpha v_\xi(t;\sigma)=\mathcal{D}_t^\alpha \biggl(h_\xi E_{\alpha,1}(-\mu_\xi M_{\sigma} t^\alpha)\\
        &+\int_0^t \left[f_\xi(s)+\mu_\xi(M_{\sigma}-\sigma(s))v_\xi(s;\sigma)\right](t-s)^{\alpha-1}E_{\alpha,\alpha}(-\mu_\xi M_{\sigma} (t-s)^\alpha)ds\biggr)\\
        &=\mathcal{D}_t^\alpha \biggl(h_\xi E_{\alpha,1}(-\mu_\xi M_{\sigma} t^\alpha)\biggr)\\
        &+\mathcal{D}_t^\alpha \biggl(\int_0^t \left[f_\xi(s)+\mu_\xi(M_{\sigma}-\sigma(s))v_\xi(s;\sigma)\right](t-s)^{\alpha-1}E_{\alpha,\alpha}(-\mu_\xi M_{\sigma} (t-s)^\alpha)ds\biggr),
    \end{aligned}
\end{equation}
for all $t\in[0,T].$ We calculate each term of \eqref{Dau 0} separately.
According to \cite{LG99}, we have
\begin{equation}\label{Dau 1}
    \begin{aligned}
        \mathcal{D}_t^\alpha \biggl(h_\xi E_{\alpha,1}(-\mu_\xi M_{\sigma} t^\alpha)\biggr)=-\mu_\xi M_{\sigma} (h_\xi E_{\alpha,1}(-\mu_\xi M_{\sigma} t^\alpha)).
    \end{aligned}
\end{equation}
Using \cite[(3.8)]{SY11}, we have
\begin{equation}\label{Dau 2}
    \begin{aligned}
      &\mathcal{D}_t^\alpha \biggl(\int_0^t \left[f_\xi(s)+\mu_\xi(M_{\sigma}-\sigma(s))v_\xi(s;\sigma)\right](t-s)^{\alpha-1}E_{\alpha,\alpha}(-\mu_\xi M_{\sigma} (t-s)^\alpha)ds\biggr)\\ 
      &=-\mu_\xi M_{\sigma} \int_0^t \left[f_\xi(s)+\mu_\xi(M_{\sigma}-\sigma(s))v_\xi(s;\sigma)\right](t-s)^{\alpha-1}E_{\alpha,\alpha}(-\mu_\xi M_{\sigma} (t-s)^\alpha)ds\\
      &+f_\xi(t)+\mu_\xi(M_{\sigma}-\sigma(t))v_\xi(t;\sigma).
    \end{aligned}
\end{equation}
Substituting \eqref{Dau 1}, \eqref{Dau 2} into \eqref{Dau 0} and taking into account \eqref{Volt int}, we get
\begin{equation*}
    \begin{aligned}
        &\mathcal{D}_t^\alpha v_\xi(t;\sigma)=-\mu_\xi M_{\sigma}\biggl(h_\xi E_{\alpha,1}(-\mu_\xi M_{\sigma} t^\alpha)\\
        &+\int_0^t \left[f_\xi(s)+\mu_\xi(M_{\sigma}-\sigma(s))v_\xi(s;\sigma)\right](t-s)^{\alpha-1}E_{\alpha,\alpha}(-\mu_\xi M_{\sigma} (t-s)^\alpha)ds\biggr)\\
        &+f_\xi(t)+\mu_\xi(M_{\sigma}-\sigma(t))v_\xi(t;\sigma)\\
        &=-\mu_\xi M_{\sigma} v_\xi(t;\sigma)+f_\xi(t)+\mu_\xi(M_{\sigma}-\sigma(t))v_\xi(t;\sigma)\\
        &=f_\xi(t)-\mu_\xi \sigma(t) v_\xi(t;\sigma),
    \end{aligned}
\end{equation*}
for all $t\in[0,T].$
Since $f_\xi,\,v_\xi(\cdot;\sigma)\in C[0,T],$ we have $\mathcal{D}_t^\alpha v_\xi\in C[0,T].$ Hence, for each $\xi\in\mathcal{I}$, there exists a unique continuous solution $v_\xi(t;\sigma)$ of \eqref{FODE} with $\mathcal{D}_t^\alpha v_\xi\in C[0,T].$ 
Due to the assumption that the system of eigenfunctions $\{\omega_\xi\}_{\xi\in\mathcal{I}}$ are an orthonormal basis in $\mathcal{H}$, we can write $v(t)$ in the form \eqref{EXPANTION u}. With the spectral representation \eqref{EXPANTION u}, the existence and uniqueness of each $v_\xi(t;\sigma)$ lead to the existence and uniqueness of the generalized solution $v(t;\sigma)$ of Problem \ref{P:D}. This completes the proof.
\end{proof}

\subsubsection{Regularity of $v(t;\sigma)$}\label{S:regularity} 
For our convenience, we first formulate the following corollary of Lemma \ref{L:E1E2}.

\begin{cor}\label{L:E1E2 ur}
For $0<\alpha<1,$ the Mittag-Leffler type function $E_{\alpha,\alpha}(-\mu_\xi m_{\sigma} t^\alpha)$ satisfies
$$0\leq E_{\alpha,\alpha}(-\mu_\xi m_{\sigma} t^\alpha)\leq E_{\alpha,\alpha}(-\inf_{\xi\in\mathcal{I}}\mu_\xi m_{\sigma} t^\alpha)\leq\frac{1}{\Gamma(\alpha)},\; t>0.$$
\end{cor}

Now, we establish the property of $v_\xi(t;\sigma)$ that is of significant importance in analyzing the regularity of $v(t;\sigma).$
\begin{lemma}\label{Cor:1}

Suppose that  $v_\xi(t;\sigma)$ is a unique solution of the problem \eqref{FODE} for each $\xi\in\mathcal{I}$. Then following statements hold:

$(a)$ If $f_\xi(t)\leq 0$ on $[0,T]$ and $h_\xi\leq 0,$ then we have $v_\xi(t;\sigma)\leq 0$ on $[0,T],\;\xi\in\mathcal{I}.$

$(b)$ If $f_\xi(t)\geq 0$ on $[0,T]$ and $h_\xi\geq 0,$ then we have $v_\xi(t;\sigma)\geq 0$ on $[0,T],\;\xi\in\mathcal{I}.$
\end{lemma}
\begin{proof}

Let $\xi\in\mathcal{I}$ be such that $\mu_\xi>0.$ 
For each $\sigma\in \mathcal{C}^+[0,T]$ and such $\xi\in\mathcal{I},$ Theorem \ref{Th:DEU} gives $v_\xi(\cdot;\sigma)\in C[0,T],$ which leads to the existence of $t_0\in [0,T]$ such that $v_\xi(t_0;\sigma)=\max_{t\in[0,T]}v_\xi(t;\sigma).$

If $t_0=0,$ then we have $v_\xi(t;\sigma)\leq v_\xi(0;\sigma)=h_\xi\leq 0.$ On other hand, if $t_0\in(0,T],$ using Lemma \ref{L:max+min}, we can deduce that $\mathcal{D}_t^\alpha v_\xi(t_0;\sigma)\geq 0,$ which implies $\mu_\xi \sigma(t_0)v_\xi(t_0;\sigma)=-\mathcal{D}_t^\alpha v_\xi(t_0;\sigma)+f_\xi(t_0)\leq 0.$ Since $\sigma>0$ on $[0,T],$ it follows $v_\xi(t_0;\sigma)\leq 0.$ Therefore, based on the definition of $t_0,$ we can conclude that $v_\xi(t;\sigma)\leq 0.$

Let $\eta\in\mathcal{I}$ be such that $\mu_\eta=0.$ By \eqref{FODE}, for each such $\eta\in\mathcal{I}$ and for $t\in[0,T]$ we have \begin{equation}\label{11}\mathcal{D}_t^\alpha v_\eta(t;\sigma)=f_\eta(t),\;v_\eta(0;\sigma)=h_\eta.\end{equation}
Applying $I_t^\alpha$ to \eqref{11} and taking into account Lemma \ref{I0D} we have the solution of \eqref{11} in the following form
$$v_\eta(t;\sigma)=h_\eta+I_t^\alpha f_\eta,$$ which yields $v_\eta(t;\sigma)\leq 0.$ 

For the case of $(b),$ let us consider $\overline{v}_\xi(t;\sigma)=-v_\xi(t;\sigma),\;\forall\xi\in\mathcal{I},$ then repeating the above proving procedure for $\overline{v}_\xi(t;\sigma),$ we show $\overline{v}_\xi(t;\sigma)\leq 0,$ for all $\xi\in\mathcal{I}$ that is, $v_\xi(t;\sigma)\geq 0.$
\end{proof}

Let us denote by $v^h(t;\sigma)$ the solution of the
homogeneous equation \eqref{EQ:Frac Pseudo} with non-zero initial condition \eqref{CON:IN}, and by $v^f(t;\sigma)$ the solution of the inhomogeneous equation \eqref{EQ:Frac Pseudo} with zero initial condition. Then in view of Theorem \ref{Th:DEU} the generalized solution $v(t;\sigma)$ of Problem \ref{P:D} can be written as $v(t;\sigma)=v^f(t;\sigma)+v^h(t;\sigma).$ This together with \eqref{EXPANTION u} imply that $v_\xi(t;\sigma)=v_\xi^f(t;\sigma)+v_\xi^h(t;\sigma),\;\xi\in\mathcal{I},$ where
$v_\xi^f(t;\sigma)$ and $v_\xi^h(t;\sigma)$ satisfy the following fractional equations
\begin{equation}\label{FODE ur}
    \mathcal{D}_t^\alpha v_\xi^f(t;\sigma)+\mu_\xi \sigma(t)v_\xi^f(t;\sigma)=f_\xi(t),\;v_\xi^f(0;\sigma)=0,\;\xi\in\mathcal{I};
\end{equation}
\begin{equation}\label{FODE ui}
    \mathcal{D}_t^\alpha v_\xi^h(t;\sigma)+\mu_\xi \sigma(t)v_\xi^h(t;\sigma)=0,\;v_\xi^h(0;\sigma)=h_\xi,\;\xi\in\mathcal{I}.
\end{equation}

Now we are going to prove the regularity of $v^{f}(t;\sigma).$ For this,
let us denote
\begin{equation}\label{f+-}
f_\xi^{max}(t)=\max\{0,f_\xi(t)\} \;\text{and}\;f_\xi^{min}(t)=\min\{0,f_\xi(t)\},\; {t}\in[0,T] \;\xi\in\mathcal{I}.
\end{equation}
It is clear that $f_\xi=f_\xi^{max}+f_\xi^{min}.$  

\begin{lemma} Let Assumption \ref{A:D} $(b)$ be satisfied. Then for each $\xi\in \mathcal{I}$ the coefficient functions $f_\xi^{max}(t),f_\xi^{min}(t)$ belong to $C[0,T].$  
\end{lemma}
\begin{proof} In view of \eqref{f+-}, we can write $f_\xi^{max}(t)$ and $f_\xi^{min}(t)$ in the following forms for each $\xi\in \mathcal{I}$ and for all $t\in[0,T]:$
\begin{equation}\label{f+-1}f_\xi^{max}(t)=\frac{f_\xi(t)+|f_\xi(t)|}{2},\;f_\xi^{min}(t)=\frac{f_\xi(t)-|f_\xi(t)|}{2}.\end{equation} 
Using \eqref{f+-1} we have the following inequality for all $t_1,\,t_2\in[0,T]:$ 
\begin{equation*}\label{f+-f+}\begin{split}
 |f_\xi^{max}(t_1)-f_\xi^{max}(t_2)|&\leq\left|\frac{f_\xi(t_1)+|f_\xi(t_1)|}{2}-\frac{f_\xi(t_2)+|f_\xi(t_2)|}{2}\right|\\
  &\leq \frac{1}{2}\big|f_\xi(t_1)-f_\xi(t_2)\big|+\frac{1}{2}\big||f_\xi(t_1)|-|f_\xi(t_2)|\big|\\
  &\leq|f_\xi(t_1)-f_\xi(t_2)|.
\end{split}
\end{equation*}
This implies with Lemma \ref{L:con f} that $f_\xi^{max}\in C[0,T].$ Similarly we can prove that $f_\xi^{min}\in C[0,T].$
\end{proof}
\begin{lemma}\label{q_a>a ur}
 Let Assumption \ref{A:D} $(I)$ hold true. Then we have
 $$|v^f_\xi(t;\sigma)|\leq \int_0^t |f_\xi(s)|(t-s)^{\alpha-1} E_{\alpha,\alpha}(-\mu_\xi m_{\sigma} (t-s)^\alpha)ds\;\text{on}\; \;[0,T],\;\xi\in\mathcal{I}.$$ 
\end{lemma}
\begin{proof}
In view of \eqref{FODE ur} for each $\xi\in\mathcal{I},$ $v^f(t;m_{\sigma})$ and $v^f(t;\sigma)$ satisfy the following system of fractional differential equations:
$$
\begin{cases}
&\mathcal{D}_t^\alpha v_\xi^f(t;m_{\sigma})+\mu_\xi m_{\sigma} v_\xi^f(t;m_{\sigma})=f_\xi(t);\\
&\mathcal{D}_t^\alpha v_\xi^f(t;\sigma)+\mu_\xi \sigma(t)v_\xi^f(t;\sigma)=f_\xi(t);\\
     &v_\xi^f(0;m_{\sigma})=v_\xi^f(0;\sigma)=0.
\end{cases}
$$
Then the difference $v_\xi^f(t;m_{\sigma})-v_\xi^f(t;\sigma)$ satisfies
$$\mathcal{D}_t^\alpha (v_\xi^f(t;m_{\sigma})-v_\xi^f(t;\sigma))+\mu_\xi m_{\sigma} (v_\xi^f(t;m_{\sigma})-v_\xi^f(t;\sigma))=\mu_\xi v_\xi^f(t;\sigma)(\sigma(t)-m_{\sigma}),$$
$$\;v_\xi^f(0;m_{\sigma})-v_\xi^f(0;\sigma)=0,$$
for each $\xi\in \mathcal{I}.$ Assumption \ref{A:D} $(I)$ together with Lemma \ref{Cor:1} give us for each $\xi\in\mathcal{I}:$
\begin{equation}\label{ura2<ura1}\begin{cases}
     &  0\leq v_\xi^{f,max}(t;\sigma)\leq v_\xi^{f,max}(t;m_{\sigma}),\;\text{on}\;[0,T];\\
     & v_\xi^{f,min}(t;m_{\sigma})\leq v_\xi^{f,min}(t;\sigma)\leq 0,\;\text{on}\;[0,T],
\end{cases}\end{equation}
where $v_\xi^{max}(t;\sigma),\; v_\xi^{f,max}(t;m_{\sigma}),$ $v_\xi^{min}(t;\sigma)$ and $v_\xi^{f,min}(t;m_{\sigma})$
satisfy 
\begin{equation*}
    \mathcal{D}_t^\alpha v_\xi^{f,max}(t;\sigma)+\lambda_\xi \sigma(t)v_\xi^{f,max}(t;\sigma)=f_\xi^{max}(t),\;v_\xi^{f,max}(0;\sigma)=0,\;\xi\in\mathcal{I};
\end{equation*}
\begin{equation*}
    \mathcal{D}_t^\alpha v_\xi^{f,max}(t;m_\sigma)+\lambda_\xi m_\sigma v_\xi^{f,max}(t;m_\sigma)=f_\xi^{max}(t),\;v_\xi^{f,max}(0;m_\sigma)=0,\;\xi\in\mathcal{I};
\end{equation*}
\begin{equation*}
    \mathcal{D}_t^\alpha v_\xi^{f,min}(t;\sigma)+\lambda_\xi \sigma(t)v_\xi^{f,min}(t;\sigma)=f_\xi^{min}(t),\;v_\xi^{f,min}(0;\sigma)=0,\;\xi\in\mathcal{I},
\end{equation*}
and
\begin{equation*}
    \mathcal{D}_t^\alpha v_\xi^{f,min}(t;m_\sigma)+\lambda_\xi m_\sigma v_\xi^{f,min}(t;m_\sigma)=f_\xi^{min}(t),\;v_\xi^{f,min}(0;m_\sigma)=0,\;\xi\in\mathcal{I},
\end{equation*}
respectively. According to \cite{LG99}, the solutions $v_\xi^f(t;m_\sigma),\,v_\xi^{f,max}(t;m_\sigma)$ and $v_\xi^{f,min}(t;m_\sigma)$ can be presented for each $\xi\in\mathcal{I}$ by the following formulas
\begin{equation}\label{uf} v_\xi^f(t;m_{\sigma})=\int_0^t f_\xi(s)(t-s)^{\alpha-1} E_{\alpha,\alpha}(-\mu_\xi m_{\sigma} (t-s)^\alpha)ds,\;\forall t\in[0,T],\end{equation}
\begin{equation}\label{ufmax} v_\xi^{f,max}(t;m_{\sigma})=\int_0^t f_\xi^{max}(s)(t-s)^{\alpha-1} E_{\alpha,\alpha}(-\mu_\xi m_{\sigma} (t-s)^\alpha)ds,\;\forall t\in[0,T],\end{equation}
and
\begin{equation}\label{ufmin} v_\xi^{f,min}(t;m_{\sigma})=\int_0^t f_\xi^{min}(s)(t-s)^{\alpha-1} E_{\alpha,\alpha}(-\mu_\xi m_{\sigma} (t-s)^\alpha)ds,\;\forall t\in[0,T].\end{equation}
In view of Corollary \ref{L:E1E2 ur}, we know that $E_{\alpha,\alpha}(-\mu_\xi m_{\sigma} t^\alpha)\geq 0, \;\forall t>0.$ Using this and taking into  account \eqref{ura2<ura1}, \eqref{ufmax} and \eqref{ufmin}, we deduce for each $\xi\in\mathcal{I}:$
$$|v^{f,max}_\xi(t;\sigma)|\leq \int_0^t f_\xi^{max}(s)(t-s)^{\alpha-1} E_{\alpha,\alpha}(-\mu_\xi m_{\sigma} (t-s)^\alpha)ds,\;\forall t\in[0,T],$$
and 
$$|v^{f,min}_\xi(t;\sigma)|\leq \int_0^t |f_\xi^{min}(s)|(t-s)^{\alpha-1} E_{\alpha,\alpha}(-\mu_\xi m_{\sigma} (t-s)^\alpha)ds,\;\forall t\in[0,T].$$
These together with $v_\xi^f(t;\sigma)=v_\xi^{f,max}(t;\sigma)+v_\xi^{f,min}(t;\sigma)$ gives us
 \begin{equation*}
        \begin{split}
            |v^f_\xi(t;\sigma)|&\leq |v^{f,max}_\xi(t;\sigma)|+|u^{f,\min}_\xi(t;\sigma)|\\
            &\leq\int_0^t (f_\xi^{max}(s)+|f_\xi^{min}(s)|)(t-s)^{\alpha-1} E_{\alpha,\alpha}(-\lambda_\xi m_{\sigma} (t-s)^\alpha)ds\\
            &=\int_0^t (f_\xi^{max}(s)-f_\xi^{min}(s))(t-s)^{\alpha-1} E_{\alpha,\alpha}(-\lambda_\xi m_{\sigma} (t-s)^\alpha)ds\\
            &=\int_0^t |f_\xi(s)|(t-s)^{\alpha-1} E_{\alpha,\alpha}(-\lambda_\xi m_{\sigma} (t-s)^\alpha)ds,
        \end{split}
    \end{equation*}
where the last equality is obtained by using \eqref{f+-1}.
This completes the proof.
\end{proof}

Now, we are ready to present the main lemma about regularity of $v^f(t;\sigma).$

\begin{lemma}\label{Cor:regularity ur}
Let Assumption \ref{A:D} $(I)$ hold.

$i)$ If $f\in C^\alpha([0,T];\mathcal{H}),$ then
$$\|v^f\|_{C([0,T];\mathcal{H}^1)}+\|\mathcal{D}_t^\alpha v^f\|_{C([0,T];\mathcal{H})}\leq C (T^\alpha+1)\|f\|_{C^\alpha([0,T];\mathcal{H})}.$$

$ii)$ If $f\in C([0,T];\mathcal{H}^\frac{1}{2}),$ then
$$\|v^f\|_{C([0,T];\mathcal{H}^1)}+\|\mathcal{D}_t^\alpha v^f\|_{C([0,T];\mathcal{H})}\leq C(T^{\alpha}+T^\frac{\alpha}{2}+1)\|f\|_{C([0,T];\mathcal{H}^\frac{1}{2})}.$$

$iii)$ If $f\in C([0,T];\mathcal{H}^\frac{1}{2})$ and $\inf_{\xi\in\mathcal{I}}\mu_\xi>0,$ then
$$\|v^f\|_{C([0,T];\mathcal{H}^1)}+\|\mathcal{D}_t^\alpha v^f\|_{C([0,T];\mathcal{H})}\leq C (\frac{1}{\sqrt{\inf_{\xi\in\mathcal{I}}\mu_\xi}}+1)\|f\|_{C([0,T];\mathcal{H}^\frac{1}{2})},$$
where $C>0$ does not depend on $T.$
\end{lemma}
\begin{proof}

First, we prove Lemma \ref{Cor:regularity ur} $i).$ For each $t\in[0,T],$ we have 
\begin{equation}\label{main}
    \begin{split}
        \|v^f(t;\sigma)\|^2_{\mathcal{H}^1}&=\sum_{\xi\in\mathcal{I}}|(1+\mu_\xi) v^f_\xi(t;\sigma)|^2\\
        &\leq C\sum_{\xi\in\mathcal{I}}|v^f_\xi(t;\sigma)|^2+C\sum_{\xi\in\mathcal{I}}|\mu_\xi v^f_\xi(t;\sigma)|^2\\
        &\leq C\sum_{\xi\in\mathcal{I}}\left|\int_0^t |f_\xi(\tau)| (t-\tau)^{\alpha-1}E_{\alpha,\alpha}(-\mu_\xi m_{\sigma} (t-\tau)^\alpha)d\tau\right|^2\\
        &+C\sum_{\xi\in\mathcal{I}}\left|\mu_\xi\int_0^t |f_\xi(\tau)| (t-\tau)^{\alpha-1}E_{\alpha,\alpha}(-\mu_\xi m_{\sigma} (t-\tau)^\alpha)d\tau\right|^2,
\end{split}\end{equation}
where the last inequality is obtained by using Lemma \ref{q_a>a ur}. Below we estimate each term of \eqref{main} separately.
Corollary \ref{L:E1E2 ur} gives us
\begin{equation}\label{2f}
    \begin{split}
     &\sum_{\xi\in\mathcal{I}}\left|\int_0^t |f_\xi(\tau)| (t-\tau)^{\alpha-1}E_{\alpha,\alpha}(-\mu_\xi m_{\sigma} (t-\tau)^\alpha)d\tau\right|^2\\
     &\leq \frac{1}{(\Gamma(\alpha))^2}\sum_{\xi\in\mathcal{I}}\left|\int_0^t |f_\xi(\tau)| (t-\tau)^{\alpha-1}d\tau\right|^2.\\
\end{split}\end{equation}
Using the H\"older inequality for $\left|\int_0^t |f_\xi(\tau)| (t-\tau)^{\alpha-1}d\tau\right|^2$, we obtain
\begin{equation*}
    \begin{split}
     &\sum_{\xi\in\mathcal{I}}\left|\int_0^t |f_\xi(\tau)| (t-\tau)^{\alpha-1}d\tau\right|^2\\&\leq\sum_{\xi\in\mathcal{I}}\left(\int_0^t|f_\xi(\tau)|^2 (t-\tau)^{\alpha-1}d\tau\right)
     \left(\int_0^t (t-\tau)^{\alpha-1}d\tau\right)\\
     &=\left(\int_0^t \left\{\sum_{\xi\in\mathcal{I}}|f_\xi(\tau)|^2\right\} (t-\tau)^{\alpha-1}d\tau\right)
     \left(\int_0^t (t-\tau)^{\alpha-1}d\tau\right)\\
     &\leq \|f\|_{C([0,T];\mathcal{H})}^2\left(\int_0^t (t-\tau)^{\alpha-1}d\tau\right)^2\leq CT^{2\alpha}\|f\|_{C([0,T];\mathcal{H})}^2. 
    \end{split}
\end{equation*}
This together with \eqref{2f} gives
\begin{equation}\label{1f}
    \begin{split}
     \sum_{\xi\in\mathcal{I}}\left|\int_0^t |f_\xi(\tau)| (t-\tau)^{\alpha-1}E_{\alpha,\alpha}(-\mu_\xi m_{\sigma} (t-\tau)^\alpha)d\tau\right|^2
     \leq CT^{2\alpha}\|f\|_{C([0,T];\mathcal{H})}^2.
\end{split}\end{equation}
And for the second term of \eqref{main}, we have
\begin{equation*}
    \begin{split}
        &\sum_{\xi\in\mathcal{I}}\left|\mu_\xi\int_0^t |f_\xi(\tau)| (t-\tau)^{\alpha-1}E_{\alpha,\alpha}(-\mu_\xi m_{\sigma} (t-\tau)^\alpha)d\tau\right|^2\\
        &\leq C\sum_{\xi\in\mathcal{I}}\left|\mu_\xi\int_0^t |f_\xi(\tau)-f_\xi(t)| (t-\tau)^{\alpha-1}E_{\alpha,\alpha}(-\mu_\xi m_{\sigma} (t-\tau)^\alpha)d\tau\right|^2\\
       &+C\sum_{\xi\in\mathcal{I}}\left||f_\xi(t)|\int_0^t \mu_\xi(t-\tau)^{\alpha-1}E_{\alpha,\alpha}(-\mu_\xi m_{\sigma} (t-\tau)^\alpha)d\tau\right|^2.
\end{split}\end{equation*}
Lemma \ref{L:Der E} together with \eqref{EST: Mittag 0} gives
\begin{equation}\label{INT E} 0\leq\int_0^t \mu_\xi (t-\tau)^{\alpha-1}E_{\alpha,\alpha}(-\mu_\xi m_{\sigma} (t-\tau)^\alpha)d\tau=\frac{1}{m_{\sigma}}(1-E_{\alpha,1}(-\mu_\xi m_{\sigma} t^\alpha))<\frac{1}{m_{\sigma}}.\end{equation}
Hence,
\begin{equation}\label{ff}
    \begin{split}
        &\sum_{\xi\in\mathcal{I}}\left|\mu_\xi\int_0^t |f_\xi(\tau)| (t-\tau)^{\alpha-1}E_{\alpha,\alpha}(-\mu_\xi m_{\sigma} (t-\tau)^\alpha)d\tau\right|^2\\
        &\leq \sum_{\xi\in\mathcal{I}}\left|\mu_\xi\int_0^t |f_\xi(\tau)-f_\xi(t)| (t-\tau)^{\alpha-1}E_{\alpha,\alpha}(-\mu_\xi m_{\sigma} (t-\tau)^\alpha)d\tau\right|^2\\
       &+\frac{1}{m_{\sigma}^2}\sum_{\xi\in\mathcal{I}}\left|f_\xi(t)\right|^2.
\end{split}\end{equation}
Now, we estimate the first term of \eqref{ff}. By Lemma \ref{L:converge E} we have
$$\begin{aligned}
&\sum_{\xi\in\mathcal{I}}\left|\mu_\xi\int_0^t |f_\xi(\tau)-f_\xi(t)| (t-\tau)^{\alpha-1}E_{\alpha,\alpha}(-\mu_\xi m_{\sigma} (t-\tau)^\alpha)d\tau\right|^2\end{aligned}$$
$$\begin{aligned}&\leq \sum_{\xi\in\mathcal{I}}\int_0^t(f_\xi(\tau)-f_\xi(t))^2(t-\tau)^{-\alpha-1}d\tau\\
&\times \int_0^t (t-\tau)^{3\alpha-1}\left(\frac{C'\mu_\xi}{1+\mu_\xi m_{\sigma} (t-\tau)^\alpha}\right)^2 d\tau\end{aligned}$$
$$\begin{aligned}&=\sum_{\xi\in\mathcal{I}}\int_0^t\frac{(f_\xi(\tau)-f_\xi(t))^2}{(t-\tau)^{2\alpha}}(t-\tau)^{\alpha-1}d\tau\\
    &\times \int_0^t (t-\tau)^{\alpha-1}\frac{C'^2}{m_{\sigma}^2}\left(\frac{\mu_\xi m_{\sigma}(t-\tau)^\alpha}{1+\mu_\xi m_{\sigma} (t-\tau)^\alpha}\right)^2 d\tau\end{aligned}$$
    $$\begin{aligned}&\leq \frac{C'^2}{m_{\sigma}^2}\left(\int_0^t\frac{\sum_{\xi\in\mathcal{I}}(f_\xi(\tau)-f_\xi(t))^2}{(t-\tau)^{2\alpha}}(t-\tau)^{\alpha-1}d\tau\right)\left(\int_0^t (t-\tau)^{\alpha-1}d\tau\right)\end{aligned}$$
    $$\begin{aligned}&\leq \frac{C'^2}{m_{\sigma}^2}\left(\int_0^t\frac{\|f(\tau)-f(t)\|^2_\mathcal{H}}{(t-\tau)^{2\alpha}}(t-\tau)^{\alpha-1}d\tau\right)\left(\int_0^t (t-\tau)^{\alpha-1}d\tau\right)\end{aligned}$$
    $$\begin{aligned}&\leq \frac{C'^2}{m_{\sigma}^2}\|f\|_{C^\alpha([0,T];\mathcal{H})}^2\left(\int_0^t (t-\tau)^{\alpha-1}d\tau\right)^2\end{aligned}$$
    $$\begin{aligned}&\leq \frac{C'^2 T^{2\alpha}}{\alpha^2m_{\sigma}^2}\|f\|_{C^\alpha([0,T];\mathcal{H})}^2\leq C T^{2\alpha} \|f\|_{C^\alpha([0,T];\mathcal{H})}^2.\end{aligned}$$
This with \eqref{ff} gives
\begin{equation}\label{fff}
    \begin{split}
        &\sum_{\xi\in\mathcal{I}}\left|\mu_\xi\int_0^t |f_\xi(\tau)| (t-\tau)^{\alpha-1}E_{\alpha,\alpha}(-\mu_\xi m_{\sigma} (t-\tau)^\alpha)d\tau\right|^2\\
        &\leq C T^{2\alpha} \|f\|_{C^\alpha([0,T];\mathcal{H})}^2+C\|f\|_{C([0,T];\mathcal{H})}^2.
\end{split}\end{equation}
Substituting \eqref{1f}, \eqref{fff} into \eqref{main}, we have
\begin{equation}\label{star}\|v^f(t;\sigma)\|^2_{\mathcal{H}^1}\leq C T^{2\alpha}\|f\|_{C^\alpha([0,T];\mathcal{H})}^2+C\|f\|_{C([0,T];\mathcal{H})}^2,\;t\in[0,T],\end{equation}
which gives
\begin{equation*}\|v^f\|_{C([0,T];\mathcal{H}^1)}\leq C(T^\alpha+1)\|f\|_{C^\alpha([0,T];\mathcal{H})}.\end{equation*}

For $\mathcal{D}_t^\alpha v^f,$ by \eqref{FODE ur}, we have $\mathcal{D}_t^\alpha v^f=\sum_{\xi\in\mathcal{I}}\left[-\mu_\xi \sigma(t)v_\xi^f(t;\sigma)+f_\xi(t)\right]\omega_\xi.$
Then for each $t\in[0,T],$
\begin{equation}\label{1Du}\begin{split}
    \|\mathcal{D}_t^\alpha v^f\|^2_\mathcal{H}&\leq C\sum_{\xi\in\mathcal{I}}M_{\sigma}^2|\mu_\xi v_\xi^f(t;\sigma)|^2+C\sum_{\xi\in\mathcal{I}}|f_\xi(t)|^2\\
    &\leq C\sum_{\xi\in\mathcal{I}}|\mu_\xi v_\xi^f(t;\sigma)|^2+C\|f(t)\|_\mathcal{H}^2.
\end{split}\end{equation}
Using estimate \eqref{star}, we have
$$\|\mathcal{D}_t^\alpha v^f\|^2_\mathcal{H}\leq C T^{2\alpha}\|f\|_{C^\alpha([0,T];\mathcal{H})}^2+C\|f\|_{C([0,T];\mathcal{H})}^2,\;t\in[0,T],$$
which gives
$$\|\mathcal{D}_t^\alpha v^f\|_{C([0,T];\mathcal{H})}\leq C(T^\alpha+1)\|f\|_{C^\alpha([0,T];\mathcal{H})}.$$

The estimates of $v^f$ and $\mathcal{D}_t^\alpha v^f$ lead to the desired result and
complete the proof of Lemma \ref{Cor:regularity ur} $i).$

Now we on the way to proving the $ii)$ part of Lemma \ref{Cor:regularity ur}. Since we assumed that $f\in C([0,T];\mathcal{H}^\frac{1}{2}),$ we estimate the second term of \eqref{main} differently from $i)$ part of Lemma \ref{Cor:regularity ur}. Using the H\"older inequality for the second term of \eqref{main}, one can get
\begin{equation}\label{4f}
    \begin{split}
        &\sum_{\xi\in\mathcal{I}}\left|\mu_\xi\int_0^t |f_\xi(\tau)| (t-\tau)^{\alpha-1}E_{\alpha,\alpha}(-\mu_\xi m_{\sigma} (t-\tau)^\alpha)d\tau\right|^2\\
        &\leq \sum_{\xi\in\mathcal{I}}\biggl(\int_0^t |\mu_\xi^\frac{1}{2}f_\xi(\tau)|^2 (t-\tau)^{\alpha-1}E_{\alpha,\alpha}(-\mu_\xi m_{\sigma} (t-\tau)^\alpha)d\tau\biggr)\\
        &\times \biggl(\int_0^t \mu_\xi(t-\tau)^{\alpha-1}E_{\alpha,\alpha}(-\mu_\xi m_{\sigma} (t-\tau)^\alpha)d\tau\biggr)\\
        &\leq \frac{1}{m_{\sigma}}\sum_{\xi\in\mathcal{I}}\int_0^t |\mu_\xi^\frac{1}{2}f_\xi(\tau)|^2 (t-\tau)^{\alpha-1}E_{\alpha,\alpha}(-\mu_\xi m_{\sigma} (t-\tau)^\alpha)d\tau,
\end{split}\end{equation}
where the last inequality is obtained by using \eqref{INT E}. In view of Corollary \ref{L:E1E2 ur}, we get
\begin{equation*}
    \begin{split}
    &\sum_{\xi\in\mathcal{I}}\int_0^t |\mu_\xi^\frac{1}{2}f_\xi(\tau)|^2 (t-\tau)^{\alpha-1}E_{\alpha,\alpha}(-\mu_\xi m_{\sigma} (t-\tau)^\alpha)d\tau\\
        &\leq \frac{1}{\Gamma(\alpha)}\sum_{\xi\in\mathcal{I}}\int_0^t |\mu_\xi^\frac{1}{2}f_\xi(\tau)|^2 (t-\tau)^{\alpha-1} d\tau\\
        &=\frac{1}{\Gamma(\alpha)}\int_0^t \left\{\sum_{\xi\in\mathcal{I}}|\mu_\xi^\frac{1}{2}f_\xi(\tau)|^2\right\} (t-\tau)^{\alpha-1} d\tau\\
        &\leq \frac{1}{\Gamma(\alpha)}\|f\|_{C([0,T];\mathcal{H}^\frac{1}{2})}^2\int_0^t (t-\tau)^{\alpha-1}d\tau\leq \frac{1}{\Gamma(\alpha+1)}T^{\alpha}\|f\|_{C([0,T];\mathcal{H}^\frac{1}{2})}^2.
\end{split}\end{equation*}
This together with \eqref{4f} give us
\begin{equation}\label{5f}
    \begin{split}
        \sum_{\xi\in\mathcal{I}}\left|\mu_\xi\int_0^t |f_\xi(\tau)| (t-\tau)^{\alpha-1}E_{\alpha,\alpha}(-\mu_\xi m_{\sigma} (t-\tau)^\alpha)d\tau\right|^2
        \leq CT^{\alpha}\|f\|_{C([0,T];\mathcal{H}^\frac{1}{2})}^2.
\end{split}\end{equation}
Substituting \eqref{1f} and \eqref{5f} into \eqref{main}, we have
\begin{equation}\label{6f}\|v^f(t;\sigma)\|^2_{\mathcal{H}^1}\leq C T^{2\alpha}\|f\|_{C([0,T];\mathcal{H})}^2+CT^\alpha\|f\|_{C([0,T];\mathcal{H}^\frac{1}{2})}^2,\;t\in[0,T],\end{equation}
which gives
\begin{equation}\label{star1}\|v^f\|_{C([0,T];\mathcal{H}^1)}\leq C(T^{\alpha}+T^\frac{\alpha}{2})\|f\|_{C([0,T];\mathcal{H}^\frac{1}{2})}.\end{equation}
In view of \eqref{1Du} and using \eqref{6f} we have
\begin{equation*}\begin{split}
    \|\mathcal{D}_t^\alpha v^f\|^2_\mathcal{H}
    &\leq C T^{2\alpha}\|f\|_{C([0,T];\mathcal{H})}^2+CT^\alpha\|f\|_{C([0,T];\mathcal{H}^\frac{1}{2})}^2+C\|f(t)\|_\mathcal{H}^2\\
    &\leq C T^{2\alpha}\|f\|_{C([0,T];\mathcal{H}^\frac{1}{2})}^2+CT^\alpha\|f\|_{C([0,T];\mathcal{H}^\frac{1}{2})}^2+C\|f(t)\|_{\mathcal{H}^\frac{1}{2}}^2,
\end{split}\end{equation*}
which gives
\begin{equation}\label{dur}\|\mathcal{D}_t^\alpha v^f\|_{C([0,T];\mathcal{H})}\leq C(T^{\alpha}+T^\frac{\alpha}{2}+1)\|f\|_{C([0,T];\mathcal{H}^{\frac{1}{2}})}.\end{equation}

Combining the estimates \eqref{star1} and \eqref{dur} yields the claimed result and complete the proof of Lemma \ref{Cor:regularity ur} $ii).$

To complete the proof of Lemma \ref{Cor:regularity ur} now we are going to prove the $iii)$ part. For each $t\in[0,T],$ we have 
\begin{equation}\label{2ur}
    \begin{aligned}
&\|v^f(t;\sigma)\|^2_{\mathcal{H}^1}=\sum_{\xi\in\mathcal{I}}|(1+\mu_\xi) v^f(t;\sigma)|^2=\sum_{\xi\in\mathcal{I}}\left|\left(1+\frac{1}{\mu_\xi}\right)\mu_\xi v^f(t;\sigma)\right|^2\\
        &\leq\sum_{\xi\in\mathcal{I}}\left|\left(1+\frac{1}{\inf_{\xi\in\mathcal{I}}\mu_\xi}\right)\mu_\xi v^f(t;\sigma)\right|^2
        \leq \left(1+\frac{1}{\inf_{\xi\in\mathcal{I}}\mu_\xi}\right)^2\sum_{\xi\in\mathcal{I}}|\mu_\xi v^f(t;\sigma)|^2\\
        &\leq \left(1+\frac{1}{\inf_{\xi\in\mathcal{I}}\mu_\xi}\right)^2\sum_{\xi\in\mathcal{I}}\left|\mu_\xi\int_0^t |f_\xi(\tau)| (t-\tau)^{\alpha-1}E_{\alpha,\alpha}(-\mu_\xi m_{\sigma} (t-\tau)^\alpha)d\tau\right|^2\\
        &\leq C\sum_{\xi\in\mathcal{I}}\left|\mu_\xi\int_0^t |f_\xi(\tau)| (t-\tau)^{\alpha-1}E_{\alpha,\alpha}(-\mu_\xi m_{\sigma} (t-\tau)^\alpha)d\tau\right|^2.
        \end{aligned}\end{equation}

Using the H\"older inequality for $\left|\mu_\xi\int_0^t |f_\xi(\tau)| (t-\tau)^{\alpha-1}E_{\alpha,\alpha}(-\mu_\xi m_{\sigma} (t-\tau)^\alpha)d\tau\right|^2$ and taking into account Corollary \ref{L:E1E2 ur} and \eqref{INT E}, one can get
    \begin{equation*}\begin{split}
        &\sum_{\xi\in\mathcal{I}}\left|\int_0^t |f_\xi(\tau)|\mu_\xi(t-\tau)^{\alpha-1}E_{\alpha,\alpha}(-\mu_\xi m_{\sigma} (t-\tau)^\alpha)d\tau\right|^2\\
        &\leq\sum_{\xi\in\mathcal{I}}\left(\int_0^t |f_\xi(\tau)|^2\mu_\xi(t-\tau)^{\alpha-1}E_{\alpha,\alpha}(-\mu_\xi m_{\sigma} (t-\tau)^\alpha)d\tau\right)\\
        &\times \left(\int_0^t\mu_\xi (t-\tau)^{\alpha-1}E_{\alpha,\alpha}(-\mu_\xi m_{\sigma} (t-\tau)^\alpha)d\tau\right)\\
        &\leq\frac{1}{m_{\sigma}}\int_0^t \left\{\sum_{\xi\in\mathcal{I}}|\mu_\xi^\frac{1}{2}f_\xi(\tau)|^2\right\}(t-\tau)^{\alpha-1}E_{\alpha,\alpha}(-\inf_{\xi\in\mathcal{I}}\mu_\xi m_{\sigma} (t-\tau)^\alpha)d\tau\\
        &\leq\frac{1}{m_{\sigma}}\|f\|_{C([0,T];\mathcal{H}^{\frac{1}{2}})}^2\int_0^t (t-\tau)^{\alpha-1}E_{\alpha,\alpha}(-\inf_{\xi\in\mathcal{I}}\mu_\xi m_{\sigma} (t-\tau)^\alpha)d\tau.
    \end{split}
\end{equation*}
Lemma \ref{L:Der E} together with \eqref{EST: Mittag 0} gives
\begin{equation}\label{INT E inf}\begin{split} &\int_0^t (t-\tau)^{\alpha-1}E_{\alpha,\alpha}(-\inf_{\xi\in\mathcal{I}}\mu_\xi m_{\sigma} (t-\tau)^\alpha)d\tau\\
&=\frac{1}{m_{\sigma} \inf_{\xi\in\mathcal{I}}\mu_\xi}(1-E_{\alpha,1}(-\inf_{\xi\in\mathcal{I}}\mu_\xi m_{\sigma} t^\alpha))<\frac{1}{m_{\sigma} \inf_{\xi\in\mathcal{I}}\mu_\xi}.\end{split}\end{equation}
Hence,
\begin{equation*}
        \sum_{\xi\in\mathcal{I}}\left|\int_0^t |f_\xi(\tau)|\mu_\xi(t-\tau)^{\alpha-1}E_{\alpha,\alpha}(-\mu_\xi m_{\sigma} (t-\tau)^\alpha)d\tau\right|^2\leq\frac{1}{m_{\sigma}^2\inf_{\xi\in\mathcal{I}}\mu_\xi}\|f\|_{C([0,T];\mathcal{H}^{\frac{1}{2}})}^2.
\end{equation*}
This together with \eqref{2ur} give us
\begin{equation}\label{1ur}\begin{split}
        &\|v^f(t;\sigma)\|^2_{\mathcal{H}^1}\leq C\frac{1}{\inf_{\xi\in\mathcal{I}}\mu_\xi}\|f\|_{C([0,T];\mathcal{H}^{\frac{1}{2}})}^2,
    \end{split}
\end{equation}
 which implies that
\begin{equation}\label{cndp ur}\|v^f\|_{C([0,T];\mathcal{H}^1)}\leq C\frac{1}{\sqrt{\inf_{\xi\in\mathcal{I}}\mu_\xi}}\|f\|_{C([0,T];\mathcal{H}^{\frac{1}{2}})}.\end{equation}

In view of \eqref{1Du} and using \eqref{1ur}, we have
\begin{equation*}\label{2Du}\begin{split}
    \|\mathcal{D}_t^\alpha v^f\|^2_\mathcal{H}
    &\leq C\frac{1}{\inf_{\xi\in\mathcal{I}}\mu_\xi}\|f\|_{C([0,T];\mathcal{H}^{\frac{1}{2}})}^2+C\|f(t)\|_\mathcal{H}^2\\
    &\leq C\frac{1}{\inf_{\xi\in\mathcal{I}}\mu_\xi}\|f\|_{C([0,T];\mathcal{H}^{\frac{1}{2}})}^2+C\|f(t)\|_{\mathcal{H}^\frac{1}{2}}^2,
\end{split}\end{equation*}
which gives
$$\|\mathcal{D}_t^\alpha v^f\|_{C([0,T];\mathcal{H})}\leq C(\frac{1}{\sqrt{\inf_{\xi\in\mathcal{I}}\mu_\xi}}+1)\|f\|_{C([0,T];\mathcal{H}^{\frac{1}{2}})}.$$
The estimates of $v^f$ and $\mathcal{D}_t^\alpha v^f$ yield the desired result and
complete this proof of Lemma \ref{Cor:regularity ur} $iii).$
\end{proof}

Now, we consider the regularity of $v^h.$
\begin{cor}\label{L:E1E2 ui}
For $0<\alpha<1,$ the Mittag-Leffler type function $E_{\alpha,1}(-\mu_\xi m_{\sigma} t^\alpha)$ satisfies
$$0<E_{\alpha,1}(-\mu_\xi m_{\sigma} t^\alpha)\leq 1,\; t\geq 0.$$
In particular, if $t=0,$ then we have
$$E_{\alpha,1}(0)=1.$$
\end{cor}
\begin{proof}
    This proof follows directly from Lemma \ref{L:E1E2} and \eqref{EST: Mittag 0}.
\end{proof}
\begin{lemma}\label{q_a>a ui}
    Let Assumption \ref{A:D} $(I)$ hold true. Then we have
    $$|v_\xi^h(t;\sigma)|\leq |h_\xi|E_{\alpha,1}(-\mu_\xi m_{\sigma} t^\alpha),\;\text{on}\;\;[0,T],\;\xi\in\mathcal{I}.$$
\end{lemma}
\begin{proof} In view of \eqref{FODE ui} for each $\xi\in\mathcal{I},$ $v^h(t;m_{\sigma})$ and $v^h(t;\sigma)$ satisfy the following system of fractional differential equations:
$$
\begin{cases}
&\mathcal{D}_t^\alpha v_\xi^h(t;m_{\sigma})+\mu_\xi m_{\sigma} v_\xi^h(t;m_{\sigma})=0;\\
&\mathcal{D}_t^\alpha v_\xi^h(t;\sigma)+\mu_\xi \sigma(t)v_\xi^h(t;\sigma)=0;\\
     &v_\xi^h(0;m_{\sigma})=v_\xi^h(0;\sigma)=0.
\end{cases}
$$
Then the difference $v_\xi^h(t;m_{\sigma})-v_\xi^h(t;\sigma)$ satisfies
$$\mathcal{D}_t^\alpha (v_\xi^h(t;m_{\sigma})-v_\xi^h(t;\sigma))+\mu_\xi m_{\sigma} (v_\xi^h(t;m_{\sigma})-v_\xi^h(t;\sigma))=\mu_\xi v_\xi^h(t;\sigma)(\sigma(t)-m_{\sigma}),$$
$$\;v_\xi^h(0;m_{\sigma})-v_\xi^h(0;\sigma)=0,$$
for each $\xi\in \mathcal{I}.$ 
According to \cite{LG99}, we have
\begin{equation}\label{uh}
   v_\xi^h(t;m_{\sigma})=h_\xi E_{\alpha,1}(-\mu_\xi m_{\sigma} t^\alpha),\;\text{on}\;\;[0,T],\;\xi\in\mathcal{I}. 
\end{equation}
In view of Assumption \ref{A:D} $(I)$ and Lemma \ref{Cor:1}, and using \eqref{uh} we obtain for each $\xi\in\mathcal{I}:$
\begin{equation}\label{uia2<uia1}\begin{cases}
     &  0\leq v_\xi^h(t;\sigma)\leq h_\xi E_{\alpha,1}(-\mu_\xi m_{\sigma} t^\alpha),\;\text{if} \;h_\xi\geq 0;\\
     & h_\xi E_{\alpha,1}(-\mu_\xi m_{\sigma} t^\alpha)\leq v_\xi^h(t;\sigma)\leq 0,\;\text{if} \;h_\xi\leq 0.
\end{cases}\end{equation}
Therefore, due to Corollary \ref{L:E1E2 ui} we have $E_{\alpha,1}(-\mu_\xi m_{\sigma} t^\alpha)>0, \;\forall t\geq 0.$ Hence, we deduce that $$|v_\xi^h(t;\sigma)|\leq|h_\xi| E_{\alpha,1}(-\mu_\xi m_{\sigma} t^\alpha),\; \forall t\in [0,T], \;\xi\in\mathcal{I},$$ 
completing the proof.
\end{proof}
Now, we are ready to present the main lemma about regularity of $v^h(t;\sigma).$
\begin{lemma}\label{Cor:regularity ui}
Let Assumption \ref{A:D} $(I)$ and $(III)$ be satisfied. Then
$$\|v^h\|_{C([0,T];\mathcal{H}^1)}+\|\mathcal{D}_t^\alpha v^h\|_{C([0,T];\mathcal{H})}\leq C \|h\|_{\mathcal{H}^1},$$
where $C>0$ does not depend on $T.$
\end{lemma}
\begin{proof} Corollary \ref{L:E1E2 ui} and Lemma \ref{q_a>a ui} yield that
\begin{equation*}
    \begin{split}
       &\|v^h(t;\sigma)\|_{\mathcal{H}^1}^2=\sum_{\xi\in\mathcal{I}}|(1+\mu_\xi) v_\xi^h(t;\sigma)|^2\\
       &\leq\sum_{\xi\in\mathcal{I}}|(1+\mu_\xi) h_\xi E_{\alpha,1}(-\mu_\xi m_{\sigma} t^\alpha)|^2\\
       &\leq\sum_{\xi\in\mathcal{I}}|(1+\mu_\xi)h_\xi|^2=\|h\|_{\mathcal{H}^1}^2,
    \end{split}
\end{equation*}
which gives 
\begin{equation}\label{cndp ui}
    \|v^h\|_{C([0,T];\mathcal{H}^1)}\leq \|h\|_{\mathcal{H}^1}.
\end{equation}
In view of the definition of $v^h(t;\sigma),$ it satisfies the equation $\mathcal{D}_t^\alpha v^h(t;\sigma)+\sigma(t)\mathcal{M}v^h(t;\sigma)=0.$ Using this, we get
\begin{equation*}\begin{split}
\|\mathcal{D}_t^\alpha v^h(t;\sigma)\|_\mathcal{H}^2&=\|-\sigma(t)\mathcal{M} v^h(t;\sigma)\|_\mathcal{H}^2\leq M_{\sigma}^2\sum_{\xi\in\mathcal{I}}|\mu_\xi v^h_\xi(t;\sigma)|^2\\
        &\leq M_{\sigma}^2\sum_{\xi\in\mathcal{I}}|(1+\mu_\xi) v^h_\xi(t;\sigma)|^2=M_{\sigma}^2\|v^h(t;\sigma)\|_{\mathcal{H}^1}^2.
\end{split}\end{equation*}
This implies
$$\|\mathcal{D}_t^\alpha v^h\|_{C([0,T];\mathcal{H}^1)}\leq C \|h\|_{\mathcal{H}^1}.$$
Hence, it holds that
$$\|v^h\|_{C([0,T];\mathcal{H}^1)}+\|\mathcal{D}_t^\alpha v^h\|_{C([0,T];\mathcal{H}^1)}\leq C \|h\|_{\mathcal{H}^1},$$
which leads to the claimed result.
\end{proof}

\begin{proof}[Proof of Theorem \ref{thm:MainThmDP}] Taking into account Theorem \ref{Th:DEU}, Lemmas \ref{Cor:regularity ur},  \ref{Cor:regularity ui}, and using the fact that $v(t;\sigma) = v^f(t;\sigma) + v^h(t;\sigma),$ we complete the proof of Theorem \ref{thm:MainThmDP}.
\end{proof}

\section{Well-posedness of Problem \ref{P:DIP}}\label{S:Inv}

In this section, we consider Problem \ref{P:DIP} for the operator $\mathcal{M}$ with a positive discrete spectrum such that $\inf_{\xi\in\mathcal{I}}\mu_\xi>0$ and prove the existence and uniqueness of the solution to this inverse problem. Moreover, we show a continuous dependence result.  

During this section, we assume that the given data $h,\,f$ and $E$ satisfy the following assumptions:
\begin{claim}\label{A:INV} Let $\gamma$ be a constant satisfying \eqref{gamma F}.

$(I)$ $h\in\mathcal{H}^{2+\gamma}$ with $h_\xi:=(h,\omega_\xi)\geq 0$ for each $\xi\in \mathcal{I}$ whenever $F[\omega_\xi]\neq 0;$

$(II)$ $f\in C([0,T];\mathcal{H}^{\frac{3}{2}+\gamma})$ with $f_\xi(t):=(f(t),\omega_\xi)\geq 0$ on $[0,T]$ for each $\xi\in \mathcal{I}$ whenever $F[\omega_\xi]\neq 0;$

$(III)$ $\exists \eta\in \mathcal{I}$ such that $h_\eta>0$ and $f_\eta(t)>0$ on $[0,T]$ and $F[\omega_\eta]\neq 0;$

$(IV)$ $E\in\Psi=\left\{E\in X^\alpha[0,T]\;|\; E(t)\geq c>0,\;\mathcal{D}^\alpha_tE(t)<F[f(t)]\right\}.$
\end{claim}

For clarity, we present
the following remarks which will explain Assumption \ref{A:INV} by examples.

\begin{rem}
Let $\mathcal{H}=L^2(0,1),$ and let $\mathcal{M}v=-v_{xx},\; x\in (0,1),$ with homogeneous Dirichlet boundary condition. Then the operator has the eigensystem $\{(\pi k)^2, \sqrt{2}\sin{k\pi x}\}_{k\in\mathbb{N}}.$
Let $F[v(t,\cdot)]:=\int_{0}^1 v(t,x)dx.$
Since 
\begin{equation}\label{Example int}F[\omega_k]=\int_0^1 \sqrt{2}\sin{k\pi x}dx=\frac{1+(-1)^{k+1}}{k\pi}=\begin{cases}
    &\frac{2\sqrt{2}}{k\pi},\;\text{if},\;k=2n-1\;(n\in\mathbb{N});\\
    &0, \;\text{if},\;k=2n\;(n\in\mathbb{N}),
\end{cases}\end{equation} 
we have $$\sum_{k\in\mathbb{N}} |F[\omega_k]|^2<\infty.$$ From this we see that $\gamma$ in \eqref{gamma F} can be taken to be $\gamma=0.$
From \eqref{Example int} we see that in Assumption \ref{A:INV} $(I)$ and $(II)$ it is enough to assume that for odd $k=2n-1,\;(n=1,2,...),$ 
$$h_{2n-1}\geq 0\;\;\text{and}\;\;f_{2n-1}(t)\geq 0\;\;\forall t\in[0,T].$$ 
That is why in Assumption \ref{A:INV} $(I),\,(II)$ and $(III)$ we write $F[\omega_{2n-1}]\neq 0.$ \end{rem}

\begin{rem}
 Let us consider Assumption \ref{A:INV} $(I).$ If we assume that $h$ can be represented as the sum of $\omega_\xi$ functions with positive coefficients, that is, $h=\sum_{\xi\in \mathcal{J}}c_\xi\omega_\xi,$ (where all $c_\xi > 0$ and $\mathcal{J}\subset \mathcal{I}$), then Assumption \ref{A:INV} $(I)$ would be satisfied.
 Assumption \ref{A:INV} $(II)$ can be satisfied if the function $f(t)$ can also be expressed as a linear combination of $\{\omega_\xi\}_{\xi\in\mathcal{I}},$ where each coefficient is positive. Furthermore, Assumption \ref{A:INV} $(III)$ can be fulfilled by the aforementioned example if we consider $\eta$ to be an element of $\mathcal{J},$ and if there exists a $\omega_\eta$ function such that $F[\omega_\eta]\neq 0.$ In view of Assumption \ref{A:INV} $(II)$ and $(III),$ we always have $\sum_{\xi\in\mathcal{I}} f_\xi(t) F[\omega_\xi]>0.$ Taking into account this we see that any function from the set  $\left\{E\in X^\alpha[0,T]\;|\; E(t)\geq c>0,\;\mathcal{D}^\alpha_tE(t)\leq 0\right\}$ satisfy Assumption \ref{A:INV} $(IV).$
There are many functions that satisfy Assumption \ref{A:INV} $(IV).$ For example, we can take $E(t)=\frac{c}{2\pi}\exp{(-t)},\;c>0.$
\end{rem}

\begin{rem} If $\sigma(t)$ given and belongs to $\mathcal{C}^+[0,T],$ then taking into account Lemma \ref{q_a>a ur} and Assumption \ref{A:INV} $(II),$ we have
\begin{equation*}
    \begin{split}
       &\|v^f(t;\sigma)\|_{\mathcal{H}^{2+\gamma}}^2=\sum_{\xi\in\mathcal{I}}|(1+\mu_\xi)^{2+\gamma}v^f_\xi(t;\sigma)|^2\\
       &\leq \biggr(\frac{1}{\inf_{\xi\in\mathcal{I}}\mu_\xi}+1\biggl)^{2(2+\gamma)}\sum_{\xi\in\mathcal{I}}\left|\mu_\xi^{2+\gamma}\int_0^{t} f_\xi(\tau)(t-\tau)^{\alpha-1} E_{\alpha,\alpha}(-\mu_\xi m_{\sigma} (t-\tau)^\alpha)d\tau\right|^2\\
       &\leq \biggr(\frac{1}{\inf_{\xi\in\mathcal{I}}\mu_\xi}+1\biggl)^{2(2+\gamma)} \sum_{\xi\in\mathcal{I}}\left(\int_0^{t} |\mu_\xi^{\frac{3}{2}+\gamma}f_\xi(\tau)|^2(t-\tau)^{\alpha-1} E_{\alpha,\alpha}(-\mu_\xi m_{\sigma} (t-\tau)^\alpha)d\tau\right)\\
       &\times\left(\int_0^{t} \mu_\xi(t-\tau)^{\alpha-1} E_{\alpha,\alpha}(-\mu_\xi m_{\sigma} (t-\tau)^\alpha)d\tau\right)
       \end{split}
\end{equation*}
       \begin{equation*}
    \begin{split}
       &\leq \biggr(\frac{1}{\inf_{\xi\in\mathcal{I}}\mu_\xi}+1\biggl)^{2(2+\gamma)}\frac{1}{m_{\sigma}} \int_0^{t} \left\{\sum_{\xi\in\mathcal{I}}|\mu_\xi^{\frac{3}{2}+\gamma}f_\xi(\tau)|^2\right\}(t-\tau)^{\alpha-1} E_{\alpha,\alpha}(-\inf_{\xi\in\mathcal{I}}\mu_\xi m_{\sigma} (t-\tau)^\alpha)d\tau\\ 
       &\leq \biggr(\frac{1}{\inf_{\xi\in\mathcal{I}}\mu_\xi}+1\biggl)^{2(2+\gamma)}\frac{1}{m_{\sigma}} \|f\|^2_{C([0,T];\mathcal{H}^{\frac{3}{2}+\gamma})}\int_0^{t}(t-\tau)^{\alpha-1} E_{\alpha,\alpha}(-\inf_{\xi\in\mathcal{I}}\mu_\xi m_{\sigma} (t-\tau)^\alpha)d\tau\\ 
       &\leq \frac{\biggr(\frac{1}{\inf_{\xi\in\mathcal{I}}\mu_\xi}+1\biggl)^{2(2+\gamma)}}{m_{\sigma}^2 \inf_{\xi\in\mathcal{I}}\mu_\xi}\|f\|^2_{C([0,T];\mathcal{H}^{\frac{3}{2}+\gamma})},
    \end{split}
\end{equation*}
where the last inequality is derived from \eqref{INT E inf}. This gives 
\begin{equation}\label{Est:ur/}
\begin{split}
       \|v^f(t;\sigma)\|_{C([0,T];{\mathcal{H}^{2+\gamma}})}\leq \frac{\biggr(\frac{1}{\inf_{\xi\in\mathcal{I}}\mu_\xi}+1\biggl)^{2+\gamma}}{m_{\sigma} \sqrt{\inf_{\xi\in\mathcal{I}}\mu_\xi}}\|f\|_{C([0,T];\mathcal{H}^{\frac{3}{2}+\gamma})}.
    \end{split}
\end{equation}

In view of Assumption \ref{A:INV} $(II),$ and using Corollary \ref{L:E1E2 ui} and Lemma \ref{q_a>a ui}, we obtain
\begin{equation*}
    \begin{split}
       &\|v^h(t;\sigma)\|_{\mathcal{H}^{2+\gamma}}^2=\sum_{\xi\in\mathcal{I}}|(1+\mu_\xi)^{2+\gamma}v^h_\xi(t;\sigma)|^2\\
       &\leq \biggr(\frac{1}{\inf_{\xi\in\mathcal{I}}\mu_\xi}+1\biggl)^{2(2+\gamma)}\sum_{\xi\in\mathcal{I}}\left|\mu_\xi^{2+\gamma}h_\xi E_{\alpha,1}(-\mu_\xi m_{\sigma} t^\alpha)\right|^2\\
       &\leq \biggr(\frac{1}{\inf_{\xi\in\mathcal{I}}\mu_\xi}+1\biggl)^{2(2+\gamma)}\sum_{\xi\in\mathcal{I}}\left|\mu_\xi^{2+\gamma}h_\xi\right|^2\\
       &\leq \biggr(\frac{1}{\inf_{\xi\in\mathcal{I}}\mu_\xi}+1\biggl)^{2(2+\gamma)}\|h\|^2_{\mathcal{H}^{2+\gamma}}.
    \end{split}
\end{equation*}
This gives us 
\begin{equation}\label{EST:u in}
\|v^h(t;\sigma)\|_{\mathcal{H}^{2+\gamma}}\leq \biggr(\frac{1}{\inf_{\xi\in\mathcal{I}}\mu_\xi}+1\biggl)^{2+\gamma}\|h\|_{\mathcal{H}^{2+\gamma}}.
\end{equation}
Combining \eqref{Est:ur/}, \eqref{EST:u in} and using $v(t;\sigma)=v^f(t;\sigma)+v^h(t;\sigma),$ one can obtain
\begin{equation}\label{EQ:convergence u/}
\|v\|_{C([0,T];\mathcal{H}^{2+\gamma})}\leq \biggr(\frac{1}{\inf_{\xi\in\mathcal{I}}\mu_\xi}+1\biggl)^{2+\gamma}\left(\|f\|_{C([0,T];\mathcal{H}^{\frac{3}{2}+\gamma})}+\|h\|_{\mathcal{H}^{2+\gamma}}\right).
\end{equation}
This implies $v\in C([0,T];\mathcal{H}^{2+\gamma}).$ Since, $F:\mathcal{H}^{1+\gamma}\rightarrow \mathbb{R},$ the $\mathcal{H}^{2+\gamma}$-regularity of $v(t;\sigma)$ implies that $F[v(t;\sigma)],\,F[\mathcal{D}_t^\alpha v(t;\sigma)]$ and  $F[\mathcal{M}v(t;\sigma)]$ are well-defined for proving the well-posedness of Problem \ref{P:DIP}.
\end{rem}

For Problem \ref{P:DIP}, we now present the main theorem of this section.

\begin{thm}\label{Th:main theorem}
   Let Assumption \ref{A:INV} be satisfied. Then the inverse Problem \ref{P:DIP} is well-posed.
\end{thm}

\subsection{Proof of Theorem \ref{Th:main theorem}}
First, we prove existence of the solution $(\sigma,v)$ of Problem \ref{P:DIP} in subsection \ref{S:existence}. Second, we show continuous dependence  of $(\sigma,v)$ in subsection \ref{S:continuous}. Then we prove uniqueness of $(\sigma,v)$ in subsection \ref{S:uniqueness}. 

\subsubsection{Existence of $(\sigma,v)$}\label{S:existence}
The existence of $\sigma\in \mathcal{C}^+[0; T]$ with Theorem \ref{Th:DEU} give us the existence of the solution $v(t;\sigma).$ These results tell us the existence of the pair of functions $(\sigma,v),$ which is the solution of Problem \ref{P:DIP}.
In view of this conclusion, in this subsection, we first, deal with showing the existence of the coefficient $\sigma(t)$ from the additional data
$$F[v(t)]=E(t),\;t\in[0,T].$$ 
The  process of showing the existence of the coefficient $\sigma(t)$ follows the following steps:

    \textbf{Step 1,} we introduce the operator $P$ and reduce Problem \ref{P:DIP} to an operator equation for $\sigma(t),$ i.e
    $$P[\sigma(t)]=\sigma(t),\;t\in[0,T].$$
    Then we determine the domain $D$ of $P$ and formulate as lemmas the properties of the operator $P;$
    
    \textbf{Step 2,} we prove the existence of the fixed point $\sigma^*(t)$ of $P$ in $D,$ by using Shauder's fixed point Theorem \ref{Th:ShF}.

The functional $F$ is linear and bounded on $\mathcal{H}^{1+\gamma}.$ We assumed the linearity and boundedness of $F$, when we introduce the additional condition \eqref{CON:ADD}. Acting by the functional $F$ on \eqref{EXPANTION u} 
we get 
\begin{equation}
\label{FU}F[v(t)]=\sum_{\xi\in\mathcal{I}}v_\xi(t;\sigma)F[\omega_\xi].
\end{equation}
The linearity and boundedness of $F$ on $\mathcal{H}^{1+\gamma}$ in view of Theorem \ref{LFB} give us continuity of $F$ on $\mathcal{H}^{1+\gamma}.$ Here we use the continuity and linearity of $F$ to put functional $F$ under the sum and to get \eqref{FU}. 
Applying the operator $\mathcal{D}_t^\alpha$ to \eqref{FU}, we have the following
\begin{equation}\label{DFU}
\mathcal{D}_t^\alpha F[v(t)]=\sum_{\xi\in\mathcal{I}}\mathcal{D}_t^\alpha v_\xi(t;\sigma)F[\omega_\xi].
\end{equation}
Assumption \ref{A:INV} $(IV)$ together with \eqref{CON:ADD} allow the acting by $\mathcal{D}_t^\alpha$ to $F[v(t)]$ and guarantee its meaningfulness.
Acting by the operator $\mathcal{D}_t^\alpha$ to \eqref{EXPANTION u} under Theorem \ref{thm:MainThmDP} we have
\begin{equation}\label{DU}
\mathcal{D}_t^\alpha v(t;\sigma)=\sum_{\xi\in\mathcal{I}}\mathcal{D}_t^\alpha v_\xi(t;\sigma)\omega_\xi.
\end{equation}
Applying the functional $F$ to \eqref{DU} and taking into account \eqref{DFU} we have
\begin{equation}\label{change order}
    \begin{split}
        F[\mathcal{D}_t^\alpha v(t)]=\mathcal{D}_t^\alpha F[v(t)].
    \end{split}
\end{equation}
Acting by the operator $\mathcal{D}_t^\alpha$ to \eqref{CON:ADD} under Assumption \ref{A:INV} $(IV)$ and taking into account \eqref{change order}, we obtain the following
\begin{equation}\label{dif F}
F[\mathcal{D}^\alpha_t v(t)]=\mathcal{D}^\alpha_t E(t),\;t\in[0,T].
\end{equation}
Applying $F$ to both sides of \eqref{EQ:Frac Pseudo}, we get
\begin{equation*}
    \sigma(t)=\frac{F[f(t)]-F[\mathcal{D}_t^\alpha v(t)]}{F[\mathcal{M}v(t)]}.
\end{equation*}

We will now show that the expression on the right hand side is well-defined, and we analyse its properties.

Using the expansion form of $f(t)$ that is, $f(t)=\sum_{\xi\in\mathcal{I}}f_\xi(t)\omega_\xi$ and taking into account \eqref{EXPANTION u}, \eqref{dif F}, we have 
\begin{equation}\label{EQ:OPerator}
    \begin{split}
    \sigma(t)=\frac{\sum_{\xi\in\mathcal{I}}F[\omega_\xi]f_\xi(t)-\mathcal{D}_t^\alpha E(t)}{\sum_{\xi\in\mathcal{I}}\mu_\xi F[\omega_\xi]v_\xi(t;\sigma)}.
    \end{split}
\end{equation}
Showing the existence and uniqueness of the solution $\sigma(t)$ of \eqref{EQ:OPerator}, we use the Shauder fixed point theorem. For this, we rewrite equation \eqref{EQ:OPerator} in the form $\sigma(t)=P[\sigma](t),$
where
\begin{equation}\label{Operator}
    P[\sigma](t)=\frac{\sum_{\xi\in\mathcal{I}}F[\omega_\xi]f_\xi(t)-\mathcal{D}_t^\alpha E(t)}{\sum_{\xi\in\mathcal{I}}\mu_\xi F[\omega_\xi]v_\xi(t;\sigma)}.
\end{equation}

Now we are going to determine the domain of the operator $P$ in \eqref{Operator}. To this end, we will denote
$$ C_0=\inf_{t\in[0,T]}\left(\sum_{\xi\in\mathcal{I}}F[\omega_\xi]f_\xi(t)-\mathcal{D}_t^\alpha E(t)\right),$$
$$C_1=\sup_{t\in[0,T]}\left(\sum_{\xi\in\mathcal{I}}F[\omega_\xi]f_\xi(t)-\mathcal{D}_t^\alpha E(t)\right),$$
$$C_2=\inf_{\xi\in\mathcal{I}}\mu_\xi \cdot \inf_{t\in[0,T]}E(t),$$
$$C_3=\sum_{\xi\in\mathcal{I}}\mu_\xi F[\omega_\xi]h_\xi+\sup_{t\in[0,T]}\left(\sum_{\xi\in\mathcal{I}}\mu_\xi F[\omega_\xi]I_t^\alpha[f_\xi(t)]\right).$$

It is easy to verify that $0<C_i<\infty,\;(i=0,1,2,3)$ by Assumption \ref{A:INV} and the definition of $I_t^\alpha$.
Here we only need to show the convergence of the series $\sum_{\xi\in\mathcal{I}}\mu_\xi F[\omega_\xi]I_t^\alpha[f_\xi(t)].$ We have

\begin{equation*}
    \begin{split}
        \sum_{\xi\in\mathcal{I}}\mu_\xi F[\omega_\xi]I_t^\alpha[f_\xi(t)]&=\frac{1}{\Gamma(\alpha)}\sum_{\xi\in\mathcal{I}}\mu_\xi F[\omega_\xi]\int_0^t f_\xi(s)(t-s)^{\alpha-1}ds\\
        &=\frac{1}{\Gamma(\alpha)}\int_0^t (t-s)^{\alpha-1} \left\{\sum_{\xi\in\mathcal{I}}\mu_\xi F[\omega_\xi] f_\xi(s)\right\}ds\\
        &\leq\frac{1}{\Gamma(\alpha)}\int_0^t (t-s)^{\alpha-1} \left(\sum_{\xi\in\mathcal{I}}\left|\frac{F[\omega_\xi]}{\mu_\xi^\gamma}\right|^2\right)^\frac{1}{2}\left(\sum_{\xi\in\mathcal{I}}|\mu_\xi^{1+\gamma}f_\xi(s)|^2\right)^\frac{1}{2}ds\\
        &\leq C_F\frac{1}{\Gamma(\alpha)}\int_0^t (t-s)^{\alpha-1} \|f(s)\|_{\mathcal{H}^{1+\gamma}}ds\\
        &\leq C_F \|f\|_{C([0,T];\mathcal{H}^{1+\gamma})}\frac{1}{\Gamma(\alpha)}\int_0^t (t-s)^{\alpha-1}ds\\
        &\leq C_F \|f\|_{C([0,T];\mathcal{H}^{1+\gamma})}\frac{t^\alpha}{\Gamma(\alpha+1)}\leq\frac{C_F T^\alpha}{\Gamma(\alpha+1)} \|f\|_{C([0,T];\mathcal{H}^{1+\gamma})},
    \end{split}
\end{equation*}
where $C_F=\left(\sum_{\xi\in\mathcal{I}}\left|\frac{F[\omega_\xi]}{\mu_\xi^\gamma}\right|^2\right)^\frac{1}{2},$ finite by assumption \eqref{gamma F}.

We can easily see that $C_0\leq C_1$ and that the numerator of \eqref{Operator} belongs to the segment $[C_0,C_1],$ that is,
\begin{equation}\label{C0<C1}
    C_0\leq \sum_{\xi\in\mathcal{I}}F[\omega_\xi]f_\xi(t)-\mathcal{D}_t^\alpha E(t) \leq C_1.
\end{equation}
Now we show that $C_2\leq C_3$ and that the denominator of \eqref{Operator} belongs to the segment $[C_2,C_3].$  For this, first, we formulate the following corollary. 

\begin{cor}\label{cor:-<+}
Let Assumption \ref{A:INV} $(I)$ and $(II)$ hold true. If $\sigma\in \mathcal{C}^+[0,T],$ then for any  $\xi\in \mathcal{I},\; v_\xi(t;\sigma)\geq 0$ on $[0,T].$ 
\end{cor}
\begin{proof}
 In view of $v_\xi(t;\sigma)=v_\xi^h(t;\sigma)+v_\xi^f(t;\sigma),$ Assumption \ref{A:INV} $(I),(II)$ together with \eqref{ura2<ura1}, \eqref{uia2<uia1} give us $v_\xi(t;\sigma)\geq 0$ on $[0,T].$
\end{proof}
In view of Corollary \ref{cor:-<+}, we have 
\begin{equation*}\begin{aligned}
0&<\inf_{\xi\in\mathcal{I}}\mu_\xi \inf_{t\in[0,T]} E(t)=C_2=\inf_{\xi\in\mathcal{I}}\mu_\xi \inf_{t\in[0,T]}\left(\sum_{\xi\in\mathcal{I}}F[\omega_\xi]v_\xi(t;\sigma)\right)
\leq \sum_{\xi\in\mathcal{I}}\mu_\xi F[\omega_\xi]v_\xi(t;\sigma)\\
&\leq \sum_{\xi\in\mathcal{I}}\mu_\xi F[\omega_\xi](h_\xi+I_t^\alpha[f_\xi(t)])\leq \sum_{\xi\in\mathcal{I}}\mu_\xi F[\omega_\xi]h_\xi+\sup_{t\in[0,T]}\left(\sum_{\xi\in\mathcal{I}}\mu_\xi F[\omega_\xi]I_t^\alpha[f_\xi(t)]\right)=C_3,
\end{aligned}
\end{equation*}
that is,
\begin{equation}\label{C2<C3}
    C_2\leq \sum_{\xi\in\mathcal{I}}\mu_\xi F[\omega_\xi]v_\xi(t;\sigma)\leq C_3. 
\end{equation}
Here we used the following estimate: By applying the operator
$I_t^\alpha$ to both sides of equation \eqref{FODE} and using Lemma \ref{I0D}, we obtain
$$v_\xi(t;\sigma)+\mu_\xi I_t^\alpha[\sigma(t)v_\xi(t;\sigma)]=I_t^\alpha f_\xi(t) +h_\xi.$$
In view of Corollary \ref{cor:-<+}, we have $v_\xi(t;\sigma)\geq 0$ on $[0,T].$ This, combined with the positivity of $\mu_\xi,$ the positivity of $\sigma,$ and the definition of $I_t^\alpha,$ leads us to the inequality $\mu_\xi I_t^\alpha[\sigma(t)v_\xi(t;\sigma)]\geq 0.$ Given that $v_\xi(t;\sigma)\geq 0$ and $\mu_\xi I_t^\alpha[\sigma(t)v_\xi(t;\sigma)]\geq 0,$ we can conclude that $0\leq v_\xi(t;\sigma)\leq I_t^\alpha f_\xi(t) +h_\xi$ on the interval $[0,T].$ Consequently, based on \eqref{Positiveness F} and Assumption \ref{A:INV} $(I)$ and $(II),$ the following inequality holds
$$\sum_{\xi\in\mathcal{I}}\mu_\xi F[\omega_\xi]v_\xi(t;\sigma)
\leq \sum_{\xi\in\mathcal{I}}\mu_\xi F[\omega_\xi](h_\xi+I_t^\alpha[f_\xi(t)]).$$

 Using \eqref{EQ:OPerator} and taking into account \eqref{C0<C1}, \eqref{C2<C3}, we have
$$0<\frac{C_0}{C_3}\leq \sigma(t)\leq \frac{C_1}{C_2}.$$

Let us introduce now the domain of the operator $P$ as
\begin{equation*}\begin{split}
    D=\biggl\{\sigma\in \mathcal{C}^+[0,T]\big|\; &\frac{C_0}{C_3}\leq \sigma(t)\leq \frac{C_1}{C_2}\biggr\}.
\end{split}\end{equation*}

For the operator $P,$ we state the following lemmas.

\begin{lemma}
The operator $P$ well-defined on the domain $D.$
\end{lemma}
\begin{proof}
According to Theorem \ref{Th:DEU}, for every $\sigma(t)$ in the domain $D,$ there exists a unique $v_\xi(t;\sigma)$ for each $\xi\in\mathcal{I}.$ This result, in turn, establishes the existence and uniqueness of $P[\sigma].$ In view of $C_2>0$ and \eqref{C2<C3}, the denominator of $P[\sigma]$
\begin{equation}\label{Positiveness denominator K final}
\sum_{\xi\in\mathcal{I}} \mu_\xi F[\omega_\xi]v_\xi(t;\sigma)>0,\;\text{on}\; [0,T].
\end{equation}
This completes the proof.
\end{proof}
\begin{lemma}
The operator $P$ maps the domain $D$ to itself.
\end{lemma}
\begin{proof}
Let $\sigma\in D.$  The continuity of $P[\sigma]$ in $t$ can be established based on the continuity of $f_\xi,\;E$ and $v_\xi(t;\sigma)$ for each $\xi\in\mathcal{I}.$ Since the inequalities \eqref{C2<C3} do not dependent on $\sigma,$ the equality \eqref{Operator} together with \eqref{C0<C1}, \eqref{C2<C3} give us that
\begin{equation}\label{c1<Qa<c2}0<\frac{C_0}{C_3}\leq P[\sigma]\leq \frac{C_1}{C_2},\;\forall{\sigma}\in D,\end{equation}
completing the proof.
\end{proof}
 
\begin{lemma}\label{L:Bounded}
The set $P(D)=\{P[\sigma]:\,\sigma\in D\}$ is uniformly bounded. 
\end{lemma}
\begin{proof} The estimate \eqref{c1<Qa<c2} yields that
$$P[\sigma]\leq \frac{C_1}{C_2},$$
for all $\sigma\in D$ and for all $P[\sigma]\in P(D).$ This completes the proof.
\end{proof}
\begin{lemma}\label{L:equicon}
The set $P(D)=\{P[\sigma]:\,\sigma\in D\}$ is equicontinuous.
\end{lemma}
\begin{proof}
In view of \eqref{Operator} we have for any $\sigma\in D$ and $\forall t_1,t_2\in [0,T],$
\begin{equation}\label{Differ K}
    \biggl|P[\sigma](t_1)-P[\sigma](t_2)\biggr|\leq \frac{|M(t_1)-M(t_2)|}{N(t_2)}+\frac{M(t_1)|N(t_1)-N(t_2)|}{N(t_1)N(t_2)},
\end{equation}
where
$$M(t)=\sum_{\xi\in \mathcal{I}}F[\omega_\xi] f_\xi(t) - \mathcal{D}_t^\alpha E(t),$$
$$N(t)=\sum_{\xi\in \mathcal{I}}\mu_\xi F[\omega_\xi]v_\xi(t;\sigma).$$

Since the difference \eqref{Differ K} is estimated using the differences of the functions $M(t)$ and
$N(t)$, we further evaluate these differences. First, we obtain an estimate of the difference of $N(t).$
Therefore, we will deal with the function $M(t).$

Without loss of generality, let us take $t_1<t_2.$ Then by Theorem \ref{Cor} for $\xi\in\mathcal{I}$ there exists $t\in(t_1,\,t_2)$ such that
\begin{equation}\label{ut1-ut2}
    |v_\xi(t_1;\sigma)-v_\xi(t_2;\sigma)|=\frac{1}{\Gamma(\alpha+1)}|{}_{t_1}\mathcal{D}_t^\alpha v_\xi(t;\sigma)|\cdot|t_1-t_2|^\alpha.
\end{equation}
Here ${}_{t_1}\mathcal{D}_t^\alpha v_\xi(t;\sigma)$ is defined in the following way: 

Let us pick $\xi\in\mathcal{I},$ and then the definition of the Caputo derivative yields the following
\begin{align*}
{}_{t_1}\mathcal{D}_t^\alpha v_\xi(t;\sigma)&=\frac{1}{\Gamma(1-\alpha)}\int_{t_1}^t (t-\tau)^{-\alpha}\frac{d}{d\tau} v_\xi(\tau;\sigma)d\tau\\
&=\frac{1}{\Gamma(1-\alpha)}\int_{0}^t (t-\tau)^{-\alpha}\frac{d}{d\tau} v_\xi(\tau;\sigma)d\tau-\frac{1}{\Gamma(1-\alpha)}\int_{0}^{t_1} (t-\tau)^{-\alpha}\frac{d}{d\tau} v_\xi(\tau;\sigma)d\tau\\
&=\mathcal{D}_t^\alpha v_\xi(t;\sigma)-\mathcal{D}_{t_1}^\alpha v_\xi(t_1;\sigma).
\end{align*}
We note that in the case $t_1=0$ the following equality holds: $${}_{t_1}\mathcal{D}_t^\alpha v_\xi(t;\sigma)=\mathcal{D}_t^\alpha v_\xi(t;\sigma).$$
Now in view of \eqref{FODE} we express ${}_{t_1}\mathcal{D}_t^\alpha v_\xi(t;\sigma)$ for $t_1>0$ as follows 
$$\mathcal{D}_t^\alpha v_\xi(t;\sigma)+\mu_\xi \sigma(t)v_\xi(t;\sigma)=f_\xi(t),\; \text{for}\; t\in(t_1,t_2);$$
$$\mathcal{D}_{t_1}^\alpha v_\xi(t_1;\sigma)+\mu_\xi \sigma(t_1)v_\xi(t_1;\sigma)=f_\xi(t_1),\; \text{for}\; t_1\in(0,T].$$
Subtracting these equations from each other we get
$${}_{t_1}\mathcal{D}_t^\alpha v_\xi(t;\sigma)=f_\xi(t)-\mu_\xi \sigma(t)v_\xi(t;\sigma)-\left(f_\xi(t_1)-\mu_\xi \sigma(t_1)v_\xi(t_1;\sigma)\right).$$
Then in view of Assumption \ref{A:INV} $(II)$ and Corollary \ref{cor:-<+}, we have for all $t_1\geq 0$ and $t\in(t_1,t_2)$ that
\begin{align*}
  &|{}_{t_1}\mathcal{D}_t^\alpha v_\xi(t;\sigma)|\leq f_\xi(t)+\mu_\xi M_{\sigma} v_\xi(t;\sigma)+f_\xi(t_1)+\mu_\xi M_{\sigma} v_\xi(t_1;\sigma).
\end{align*}
Here and further we can take $m_{\sigma}=\frac{C_0}{C_3},\,M_{\sigma}=\frac{C_1}{C_2}.$
This together with \eqref{ut1-ut2} gives
\begin{equation*}\begin{split}
        &|N(t_1)-N(t_2)|\leq \sum_{\xi\in \mathcal{I}}\mu_\xi F[\omega_\xi]|v_\xi(t_1;\sigma)-v_\xi(t_2;\sigma)|\\&=\frac{|t_1-t_2|^\alpha}{\Gamma(\alpha+1)}\sum_{\xi\in \mathcal{I}}\mu_\xi F[\omega_\xi]\cdot |{}_{t_1}\mathcal{D}_t^\alpha v_\xi(t;\sigma)|\\
        &\leq \frac{|t_1-t_2|^\alpha}{\Gamma(\alpha+1)}\biggl(\sum_{\xi\in \mathcal{I}}\mu_\xi F[\omega_\xi]f_\xi(t)+M_{\sigma}\sum_{\xi\in \mathcal{I}}\mu_\xi^2 F[\omega_\xi] v_\xi(t;\sigma)\\
        &+\sum_{\xi\in \mathcal{I}}\mu_\xi F[\omega_\xi]f_\xi(t_1)+M_{\sigma}\sum_{\xi\in \mathcal{I}}\mu_\xi^2 F[\omega_\xi] v_\xi(t_1;\sigma)\biggr).
        \end{split}\end{equation*}
By using the Cauchy-Schwarz inequality for each term, we obtain       
\begin{equation}\begin{split}\label{N}
        &|N(t_1)-N(t_2)|\\&\leq\frac{|t_1-t_2|^\alpha}{\Gamma(\alpha+1)}\biggl[C_F\left(\sum_{\xi\in \mathcal{I}}|\mu_\xi^{1+\gamma}f_\xi(t)|^2\right)^\frac{1}{2}+C_F\left(\sum_{\xi\in \mathcal{I}}|\mu_\xi^{1+\gamma}f_\xi(t_1)|^2\right)^\frac{1}{2}\\
        &+M_{\sigma} C_F\left(\sum_{\xi\in \mathcal{I}}|\mu_\xi^{2+\gamma}v_\xi(t;\sigma)|^2\right)^\frac{1}{2}+M_{\sigma} C_F\left(\sum_{\xi\in \mathcal{I}}|\mu_\xi^{2+\gamma}v_\xi(t_1;\sigma)|^2\right)^\frac{1}{2}\biggr]\\
        &\leq \frac{|t_1-t_2|^\alpha}{\Gamma(\alpha+1)}\biggl(2C_F\|f\|_{C([0,T];\mathcal{H}^{1+\gamma})}+2M_{\sigma}C_F\|v\|_{C([0,T];\mathcal{H}^{2+\gamma})}\biggr)\\
        &\leq C_4 |t_1-t_2|^\alpha,
\end{split}\end{equation}
where
\begin{align*}
C_4&=\frac{2C_F}{\Gamma(\alpha+1)}\biggl[\|f\|_{C([0,T];\mathcal{H}^{1+\gamma})}\\
&+M_{\sigma}\biggl(\frac{1}{\inf_{\xi\in\mathcal{I}}\mu_\xi}+1\biggl)^{2+\gamma}\left(\|h\|_{\mathcal{H}^{2+\gamma}}+\frac{1}{m_{\sigma} \sqrt{\inf_{\xi\in\mathcal{I}}\mu_\xi}}\|f\|_{C([0,T];\mathcal{H}^{\frac{3}{2}+\gamma})}\right)\biggr]
  \end{align*}   
and the last inequality is obtained by using \eqref{EQ:convergence u/}.

Fix an arbitrary $\varepsilon>0.$ Since $M(t)$ is continuous in $[0,T]$, then $\exists\delta_1=\delta_1(\varepsilon),\;\forall t_1,t_2\in[0,T]\;(|t_1-t_2|<\delta_1):$
\begin{equation}\label{M}
    |M(t_1)-M(t_2)|<\frac{C_2\varepsilon}{2}.
\end{equation}

Let $$\delta=\min\left\{\delta_1(\varepsilon),\,\left(\frac{C_2^2}{2C_1 C_4}\varepsilon\right)^\frac{1}{\alpha}\right\}.$$

From \eqref{N} for $|t_1-t_2|<\delta,$ we obtain
\begin{equation}\label{Fin: N}
    |N(t_1)-N(t_2)|<\frac{C_2^2}{2C_1}\varepsilon.
\end{equation}

Substituting \eqref{M} and \eqref{Fin: N} into \eqref{Differ K}, we get
$$|P[\sigma](t_1)-P[\sigma](t_2)|<\varepsilon.$$

Therefore, the set $P(D)$ is equicontinuous. 
\end{proof}
\begin{thm}\label{existence (a,v)}
  Let us suppose that Assumption \ref{A:INV} are fulfilled. Then there exist the generalized solutions $(\sigma,v)$ of Problem \ref{P:DIP}. 
\end{thm}
\begin{proof}
In view of the Arzela-Ascoli Theorem \ref{Th:AA} and using Lemmas \ref{L:Bounded} and \ref{L:equicon} we see that $P(D)$ is relatively compact in $C[0, T].$ Moreover, $P:D\rightarrow D$ and $D$ is a closed convex subset of $C[0,T]$. According to the Shauder fixed point Theorem \ref{Th:ShF}, the equation $$P[\sigma]=\sigma,\;\sigma\in D,$$
has a solution $\sigma=\sigma^*\in D.$
This together with Theorem \ref{Th:DEU} allow us to conclude the existence of the generalized solutions $(\sigma,v)$ of Problem \ref{P:DIP}. This completes the proof.
\end{proof}

\subsubsection{Continuous dependence of $(\sigma,v)$ on the data}\label{S:continuous} In this subsection, we investigate the dependence of $(\sigma,v)$ on the data.

\begin{thm}\label{continuous}
Let Assumption \ref{A:INV} be satisfied. Then the solution $(\sigma,v)$ of the problem \eqref{EQ:Frac Pseudo}--\eqref{CON:ADD} depends continuously on the data, that is, there exist positive constants $M_4$ and $M_9,$ such that
\begin{equation}\label{a-b}
        \|\tilde{\sigma}-\sigma\|_{C[0,T]}\leq M_4\big(\|\tilde{h}-h\|_{\mathcal{H}^{2+\gamma}}+\|\tilde{f}-f\|_{C([0,T];\mathcal{H}^{\frac{3}{2}+\gamma})}+\|\tilde{E}-E\|_{X^\alpha[0,T]}\big),
\end{equation}
and
\begin{equation}\begin{split}\label{u-v}
    \|\tilde{v}-v\|_{C([0,T];\mathcal{H})}&\leq M_{9}\big(\|\tilde{h}-h\|_{\mathcal{H}^{2+\gamma}}+\|\tilde{f}-f\|_{C([0,T];\mathcal{H}^{\frac{3}{2}+\gamma})}+\|\tilde{E}-E\|_{X^\alpha[0,T]}\big),
\end{split}\end{equation}
where $(\tilde{\sigma},\tilde{v})$ is the solution of the inverse problem \eqref{EQ:Frac Pseudo}--\eqref{CON:ADD} corresponding to the set of data $\left\{\tilde{h},\tilde{f},\tilde{E}\right\},$ that satisfy Assumption \ref{A:INV}. 
\end{thm}
\begin{proof} Let $\Upsilon=\{h,f,E\}$ and $\tilde{\Upsilon}=\left\{\tilde{h},\tilde{f},\tilde{E}\right\}$ be two sets of data that satisfy Assumption \ref{A:INV}.
Let $(\sigma,v)$ and $(\tilde{\sigma},\tilde{v})$ be solutions of the inverse problem \eqref{EQ:Frac Pseudo}--\eqref{CON:ADD} corresponding to the data $\Upsilon$ and $\tilde{\Upsilon},$ i.e.
\begin{equation}\label{EQ:u inverse}
  \begin{cases}
  &\mathcal{D}_t^\alpha v(t;\sigma)+\sigma(t)\mathcal{M}v(t;\sigma)=f(t),\\
  &v(0;\sigma)=h,\\
  &F[v(t;\sigma)]=E(t),\\
  \end{cases}
\end{equation}
and
\begin{equation}\label{EQ:ap u inverse}
\begin{cases}
  &\mathcal{D}_t^\alpha\tilde{v}(t;\tilde{\sigma})+\tilde{\sigma}(t)\mathcal{M}\tilde{v}(t;\tilde{\sigma})=\tilde{f}(t),\\
  &\tilde{v}(0;\tilde{\sigma})=\tilde{h},\\
  &F[\tilde{v}(t;\tilde{\sigma})]=\tilde{E}(t),
  \end{cases}
\end{equation}
respectively. 

Subtracting equations \eqref{EQ:u inverse} and \eqref{EQ:ap u inverse} from each other we have
\begin{equation}\label{EQ:differ u-tilde u}
\begin{cases}
  &\mathcal{D}_t^\alpha w+\tilde{\sigma}(t)\mathcal{M}w(t)=g(t),\\
  &w(0)=\tilde{h}-h,\\
  &F[w(t)]=\tilde{E}(t)-E(t),
  \end{cases}
\end{equation}
where $w(t)=\tilde{v}(t;\tilde{\sigma})-v(t;\sigma)$ and $g(t)=\tilde{f}(t)-f(t)-(\tilde{\sigma}(t)-\sigma(t))\mathcal{M}v(t;\sigma).$ 

Similarly to \eqref{EXPANTION u}, we can write the solution $\tilde{v}(t;\tilde{\sigma})$ of the problem \eqref{EQ:ap u inverse} in the form $\tilde{v}(t;\tilde{\sigma})=\sum_{\xi\in\mathcal{I}}v_\xi(t;\tilde{\sigma})\omega_\xi.$ This together with \eqref{EXPANTION u} give us 
\begin{equation}\label{EXPANTION w}
    w(t)=\sum_{\xi\in\mathcal{I}}w_\xi(t;\tilde{\sigma})\omega_\xi,\;t\in [0,T],
\end{equation}
where $w_\xi(t;\tilde{\sigma})=\tilde{v}_\xi(t;\tilde{\sigma})-v_\xi(t;\sigma),$ and it satisfies the fractional equation
\begin{equation*}\label{EQ:ode w}
   \mathcal{D}_t^\alpha w_\xi(t;\tilde{\sigma})+\mu_\xi\tilde{\sigma}(t)w_\xi(t;\tilde{\sigma})=g_\xi(t),\;w_\xi(0;\tilde{\sigma})=\tilde h_\xi-h_\xi,\;\xi\in\mathcal{I}, 
\end{equation*}
where $g_\xi(t)=(g(t),\omega_\xi)_\mathcal{H}.$ 

Acting by $\mathcal{D}_t^\alpha$ on both sides of the last equation of \eqref{EQ:differ u-tilde u} and taking into account \eqref{change order}, we have
\begin{equation}\label{EQ:FD u-tilde u}
    F[\mathcal{D}_t^\alpha w(t)]=\mathcal{D}_t^\alpha\tilde{E}(t)-\mathcal{D}_t^\alpha E(t).
\end{equation}

Applying $F$ on both sides of the first equation of \eqref{EQ:differ u-tilde u}, we get
\begin{equation*}
   \tilde{\sigma}(t)-\sigma(t)=\frac{F[\tilde{f}(t)-f(t)]-F[D^\alpha_t w(t)]-\tilde{\sigma}(t)F[\mathcal{M}w(t)]}{F[\mathcal{M}v(t;\sigma)]}. 
\end{equation*}
We can write $f(t)$ and $\tilde{f}(t)$ in the expansion forms, i.e. $f(t)=\sum_{\xi\in\mathcal{I}}f_\xi(t)\omega_\xi$ and $\tilde{f}(t)=\sum_{\xi\in\mathcal{I}}\tilde{f}_\xi(t)\omega_\xi$ respectively. These together with \eqref{EXPANTION u},\eqref{EXPANTION w} and \eqref{EQ:FD u-tilde u} yield that  
\begin{equation}\begin{aligned}\label{a-a}
  &\tilde{\sigma}(t)-\sigma(t)\\
  &=\frac{\sum_{\xi\in\mathcal{I}}F[\omega_\xi](\tilde{f}_\xi(t)-f_\xi(t))-(\mathcal{D}_t^\alpha\tilde{E}-\mathcal{D}_t^\alpha E)-\tilde{\sigma}(t)\sum_{\xi\in\mathcal{I}}\mu_\xi F[\omega_\xi]w_\xi(t;\tilde{\sigma})}{\sum_{\xi\in\mathcal{I}}\mu_\xi F[\omega_\xi]v_\xi(t;\sigma)}. 
\end{aligned}\end{equation}
To estimate \eqref{a-a}, let us first estimate the series $\sum_{\xi\in\mathcal{I}}\mu_\xi F[\omega_\xi]w_\xi(t;\tilde{\sigma}).$ Since $w(t;\tilde{\sigma})$ is the solution of the problem \eqref{EQ:ode w}, in view of Lemmas \ref{q_a>a ur} and \ref{q_a>a ui}, we deduce that
\begin{equation*}
    \begin{split}
    |w_\xi(t;\tilde{\sigma})|&\leq |w_\xi^h(t;\tilde{\sigma})|+|w_\xi^f(t;\tilde{\sigma})|\\
    &\leq |\tilde{h}_\xi-h_\xi| E_{\alpha,1}(-\mu_\xi m_{\tilde{\sigma}} t^\alpha)+\int_0^t |g_\xi(s)|(t-s)^{\alpha-1} E_{\alpha,\alpha}(-\mu_\xi m_{\tilde{\sigma}} (t-s)^\alpha)ds\\
    &\leq |\tilde{h}_\xi-h_\xi| E_{\alpha,1}(-\mu_\xi m_{\tilde{\sigma}} t^\alpha)\\
    &+\int_0^t |\tilde{f}_\xi(s)-f_\xi(s)|(t-s)^{\alpha-1} E_{\alpha,\alpha}(-\mu_\xi m_{\tilde{\sigma}} (t-s)^\alpha)ds\\
        &+\int_0^t |\tilde{\sigma}(s)-\sigma(s)|\mu_\xi |v_\xi(s;\sigma)| (t-s)^{\alpha-1} E_{\alpha,\alpha}(-\mu_\xi m_{\tilde{\sigma}} (t-s)^\alpha)ds,
     \end{split}
\end{equation*}
where $m_{\tilde{\sigma}}=\frac{\inf_{t\in[0,T]}\left(\sum_{\xi\in\mathcal{I}}F[\omega_\xi]\tilde{f}_\xi(t)-\mathcal{D}_t^\alpha \tilde{E}(t)\right)}{\sup_{t\in[0,T]}\left(\sum_{\xi\in\mathcal{I}}\mu_\xi F[\omega_\xi](\tilde{h}_\xi+I_t^\alpha[\tilde{f}_\xi(t)])\right)}.$ 

Putting this into the series $\sum_{\xi\in\mathcal{I}}\mu_\xi F[\omega_\xi]w_\xi(t;\tilde{\sigma}),$ and taking into account Corollary \ref{cor:-<+}, we can get
\begin{equation*}
    \begin{split}
        &\sum_{\xi\in\mathcal{I}}\mu_\xi F[\omega_\xi]|w_\xi(t;\tilde{\sigma})|\leq \sum_{\xi\in\mathcal{I}}\mu_\xi F[\omega_\xi]|\tilde{h}_\xi-h_\xi|E_{\alpha,1}(-\mu_\xi m_{\tilde{\sigma}} t^\alpha)\\
        &+\sum_{\xi\in\mathcal{I}}F[\omega_\xi]\int_0^t|\tilde{f}_\xi(s)-f_\xi(s)|\mu_\xi(t-s)^{\alpha-1} E_{\alpha,\alpha}(-\mu_\xi m_{\tilde{\sigma}} (t-s)^\alpha)ds\\
        &+\sum_{\xi\in\mathcal{I}}\mu_\xi^2 F[\omega_\xi]\int_0^t|\tilde{\sigma}(s)-\sigma(s)|v_\xi(s;\sigma) (t-s)^{\alpha-1} E_{\alpha,\alpha}(-\mu_\xi m_{\tilde{\sigma}} (t-s)^\alpha)ds\\
        &=I_1(t)+I_2(t)+I_3(t).
        \end{split}
\end{equation*}
We estimate each of the three terms separately.

Using Corollary \ref{L:E1E2 ui} and the Cauchy-Schwarz inequality, we have
\begin{equation*}\begin{split}
I_1(t)&\leq \sum_{\xi\in\mathcal{I}}\mu_\xi F[\omega_\xi]|\tilde{h}_\xi-h_\xi|\leq \left(\sum_{\xi\in\mathcal{I}}\left|\frac{F[\omega_\xi]}{\mu_\xi^{\gamma}}\right|^2\right)^\frac{1}{2}\left(\sum_{\xi\in\mathcal{I}}\mu_\xi^{2(1+\gamma)}|\tilde{h}_\xi-h_\xi|^2\right)^\frac{1}{2}\\
&\leq C_F\left(\sum_{\xi\in\mathcal{I}}(1+\mu_\xi)^{2(1+\gamma)}|\tilde{h}_\xi-h_\xi|^2\right)^\frac{1}{2}\leq C_F\|\tilde{h}-h\|_{\mathcal{H}^{1+\gamma}}.
\end{split}\end{equation*}
Using the Cauchy-Schwarz and H\"older inequalities, we have
\begin{equation*}
    \begin{split}
        I_2(t)&\leq C_F\left[\sum_{\xi\in\mathcal{I}}\left|\mu_\xi^{\gamma}\int_0^t|\tilde{f}_\xi(s)-f_\xi(s)|\mu_\xi(t-s)^{\alpha-1} E_{\alpha,\alpha}(-\mu_\xi m_{\tilde{\sigma}} (t-s)^\alpha)ds\right|^2\right]^\frac{1}{2}\\
        &\leq C_F\biggl[\sum_{\xi\in\mathcal{I}}\left(\mu_\xi^{2\gamma}\int_0^t|\tilde{f}_\xi(s)-f_\xi(s)|^2\mu_\xi(t-s)^{\alpha-1} E_{\alpha,\alpha}(-\mu_\xi m_{\tilde{\sigma}} (t-s)^\alpha)ds\right)\\
        &\times \left(\int_0^t\mu_\xi(t-s)^{\alpha-1} E_{\alpha,\alpha}(-\mu_\xi m_{\tilde{\sigma}} (t-s)^\alpha)ds\right)\biggr]^\frac{1}{2}. 
    \end{split}
\end{equation*}
In view of \eqref{INT E inf} and Corollary \ref{L:E1E2 ur} we have
\begin{align*}
    \begin{split}
        I_2(t)&\leq C_F\biggl(\frac{1}{m_{\tilde{\sigma}}}\int_0^t\left\{\sum_{\xi\in\mathcal{I}}\mu_\xi^{2(\frac{1}{2}+\gamma)}|\tilde{f}_\xi(s)-f_\xi(s)|^2\right\}(t-s)^{\alpha-1} E_{\alpha,\alpha}(-\inf_{\xi\in\mathcal{I}}\mu_\xi m_{\tilde{\sigma}} (t-s)^\alpha)ds\biggr)^\frac{1}{2}\\
        &\leq M_1\|\tilde{f}-f\|_{C([0,T];\mathcal{H}^{\frac{1}{2}+\gamma})},
    \end{split}
\end{align*}
where $M_1=C_F\frac{1}{m_{\tilde{\sigma}}\sqrt{\inf_{\xi\in\mathcal{I}}\mu_\xi}}.$
Using Corollary \ref{L:E1E2 ur} and the Cauchy-Schwarz inequality, we have
\begin{equation*}
    \begin{split}
    I_3(t)&\leq \frac{1}{\Gamma(\alpha)}\int_0^t\left\{\sum_{\xi\in\mathcal{I}}\mu_\xi^2 F[\omega_\xi]v_\xi(s;\sigma)\right\} |\tilde{\sigma}(s)-\sigma(s)| (t-s)^{\alpha-1}ds\\
    &\leq \frac{C_F}{\Gamma(\alpha)}\int_0^t\left(\sum_{\xi\in\mathcal{I}}|\mu_\xi^{2+\gamma}v_\xi(s;\sigma)|^2\right)^\frac{1}{2} |\tilde{\sigma}(s)-\sigma(s)| (t-s)^{\alpha-1}ds\\
    &\leq \frac{C_F}{\Gamma(\alpha)}\|v\|_{C([0,T];\mathcal{H}^{2+\gamma})}\int_0^t|\tilde{\sigma}(s)-\sigma(s)| (t-s)^{\alpha-1}ds\\
    &\leq M_2\frac{1}{\Gamma(\alpha)}\int_0^t|\tilde{\sigma}(s)-\sigma(s)| (t-s)^{\alpha-1}ds,
    \end{split}
\end{equation*}
where $M_2=C_F\biggr(\frac{1}{\inf_{\xi\in\mathcal{I}}\mu_\xi}+1\biggl)^{2+\gamma}\left(\|f\|_{C([0,T];\mathcal{H}^{\frac{3}{2}+\gamma})}+\|h\|_{\mathcal{H}^{2+\gamma}}\right)$ and the last inequality is obtained by \eqref{EQ:convergence u/}.
The estimates obtained above for $I_1(t),\,I_2(t)$ and $I_3(t)$ yield   
\begin{equation*}
    \begin{split}
        &\sum_{\xi\in\mathcal{I}}\mu_\xi F[\omega_\xi]|w_\xi(t;\tilde{\sigma})|\\
    &\leq C_F\|\tilde{h}-h\|_{\mathcal{H}^{1+\gamma}}+M_1\|\tilde{f}-f\|_{C([0,T];\mathcal{H}^{\frac{1}{2}+\gamma})}\\
    &+M_2\frac{1}{\Gamma(\alpha)}\int_0^t |\tilde{\sigma}(s)-\sigma(s)|(t-s)^{\alpha-1}ds.
    \end{split}
\end{equation*}
Taking into account this and using \eqref{a-a} we obtain
\begin{align*}
  &|\tilde{\sigma}(t)-\sigma(t)|\\
  &\leq \frac{C_F\left(\sum_{\xi\in\mathcal{I}}\mu_\xi^{2\gamma}|\tilde{f}_\xi(t)-f_\xi(t)|^2\right)^\frac{1}{2}+\|\mathcal{D}_t^\alpha\tilde{E}-\mathcal{D}_t^\alpha E\|_{C[0,T]}+M_{\tilde{\sigma}}\sum_{\xi\in\mathcal{I}}\mu_\xi F[\omega_\xi]|w_\xi(t;\tilde{\sigma})|}{C_2}\\
&\leq\frac{1}{C_2}\big(C_F\|\tilde{f}-f\|_{C([0,T];\mathcal{H}^{\gamma})}+\|\mathcal{D}_t^\alpha\tilde{E}-\mathcal{D}_t^\alpha E\|_{C[0,T]}\\
&+M_{\tilde{\sigma}}C_F\|\tilde{h}-h\|_{\mathcal{H}^{1+\gamma}}+M_{\tilde{\sigma}}M_1\|\tilde{f}-f\|_{C([0,T];\mathcal{H}^{\frac{1}{2}+\gamma})}\\
    &+M_{\tilde{\sigma}}M_2\frac{1}{\Gamma(\alpha)}\int_0^t |\tilde{\sigma}(s)-\sigma(s)|(t-s)^{\alpha-1}ds\big)\\
&\leq M_3\big(\|\tilde{h}-h\|_{\mathcal{H}^{2+\gamma}}+\|\tilde{f}-f\|_{C([0,T];\mathcal{H}^{\frac{3}{2}+\gamma})}+\|\tilde{E}-E\|_{X^\alpha[0,T]}\\
&+\frac{1}{\Gamma(\alpha)}\int_0^t |\tilde{\sigma}(s)-\sigma(s)|(t-s)^{\alpha-1}ds\big),
\end{align*}
where $M_{\tilde{\sigma}}=\frac{\sup_{t\in[0,T]}\left(\sum_{\xi\in\mathcal{I}}F[\omega_\xi]\tilde{f}_\xi(t)-\mathcal{D}_t^\alpha \tilde{E}(t)\right)}{\inf_{\xi\in\mathcal{I}}\mu_\xi \cdot \inf_{t\in[0,T]}\tilde{E}(t)}$ and $M_3$ is some positive constant.
Applying Lemma \ref{L:Gronwall} to $|\tilde{\sigma}(t)-\sigma(t)|$, we have
\begin{equation*}
    \begin{split}
        |\tilde{\sigma}(t)-\sigma(t)|&\leq M_3 E_{\alpha,1}\left(M_3 t^\alpha\right)\big(\|\tilde{h}-h\|_{\mathcal{H}^{2+\gamma}}\\&+\|\tilde{f}-f\|_{C([0,T];\mathcal{H}^{\frac{3}{2}+\gamma})}+\|\tilde{E}-E\|_{X^\alpha[0,T]}\big).
    \end{split}
\end{equation*}
Then
\begin{equation}\label{cndp a}
        \|\tilde{\sigma}-\sigma\|_{C[0,T]}\leq M_4\big(\|\tilde{h}-h\|_{\mathcal{H}^{2+\gamma}}+\|\tilde{f}-f\|_{C([0,T];\mathcal{H}^{\frac{3}{2}+\gamma})}+\|\tilde{E}-E\|_{X^\alpha[0,T]}\big),
\end{equation}
where $M_4=M_3 E_{\alpha,1}\left(M_3 T^\alpha\right).$

Now, we are going to estimate the difference $\tilde{v}-v.$

In view of Lemma \ref{q_a>a ur}, we have
\begin{equation}
    \begin{split}\label{cndp 0a}
\|w^f(t;\tilde{\sigma})\|^2_{\mathcal{H}}&=\sum_{\xi\in\mathcal{I}}|w_\xi^f(t;\tilde{\sigma})|^2\\&=\sum_{\xi\in\mathcal{I}}\left|\int_0^t |g_\xi(\tau)|(t-\tau)^{\alpha-1}E_{\alpha,\alpha}(-\mu_\xi m_{\tilde{\sigma}} (t-\tau)^\alpha)d\tau\right|^2\\
&\leq \sum_{\xi\in\mathcal{I}}\left|\int_0^t |g_\xi(\tau)|(t-\tau)^{\alpha-1}E_{\alpha,\alpha}(-\inf_{\xi\in\mathcal{I}}\mu_\xi m_{\tilde{\sigma}} (t-\tau)^\alpha)d\tau\right|^2,
\end{split}\end{equation}
where the last inequality is obtained by using the Corollary \ref{L:E1E2 ur}. Using the H\"older inequality we get
\begin{equation}\begin{aligned}
    \label{cndp 1a}
&\sum_{\xi\in\mathcal{I}}\left|\int_0^t |g_\xi(\tau)|(t-\tau)^{\alpha-1}E_{\alpha,\alpha}(-\inf_{\xi\in\mathcal{I}}\mu_\xi m_{\tilde{\sigma}} (t-\tau)^\alpha)d\tau\right|^2\\
&\leq \sum_{\xi\in\mathcal{I}}\left(\int_0^t |g_\xi(\tau)|^2(t-\tau)^{\alpha-1}E_{\alpha,\alpha}(-\inf_{\xi\in\mathcal{I}}\mu_\xi m_{\tilde{\sigma}} (t-\tau)^\alpha)d\tau\right)\\
&\times \left(\int_0^t (t-\tau)^{\alpha-1}E_{\alpha,\alpha}(-\inf_{\xi\in\mathcal{I}}\mu_\xi m_{\tilde{\sigma}} (t-\tau)^\alpha)d\tau\right)\\
    &\leq \left(\int_0^t \left\{\sum_{\xi\in\mathcal{I}}|g_\xi(\tau)|^2\right\}(t-\tau)^{\alpha-1}E_{\alpha,\alpha}(-\inf_{\xi\in\mathcal{I}}\mu_\xi m_{\tilde{\sigma}} (t-\tau)^\alpha)d\tau\right)
\end{aligned}\end{equation}
\begin{equation*}\begin{aligned}
&\times \left(\int_0^t (t-\tau)^{\alpha-1}E_{\alpha,\alpha}(-\inf_{\xi\in\mathcal{I}}\mu_\xi m_{\tilde{\sigma}} (t-\tau)^\alpha)d\tau\right)\\
&\leq \|g\|_{C([0,T];\mathcal{H})}^2\left(\int_0^t(t-\tau)^{\alpha-1}E_{\alpha,\alpha}(-\inf_{\xi\in\mathcal{I}}\mu_\xi m_{\tilde{\sigma}} (t-\tau)^\alpha)d\tau\right)^2\\
&\leq \left(\frac{1}{\inf_{\xi\in\mathcal{I}}\mu_\xi m_{\tilde{\sigma}}}\right)^2\|g\|_{C([0,T];\mathcal{H})}^2,
\end{aligned}\end{equation*}
where the last inequality is obtained by using the estimate \eqref{INT E}.
Substituting \eqref{cndp 1a} into \eqref{cndp 0a}, we obtain
$$\|w^f(t;\tilde{\sigma})\|^2_{\mathcal{H}}\leq \left(\frac{1}{\inf_{\xi\in\mathcal{I}}\mu_\xi m_{\tilde{\sigma}}}\right)^2\|g\|_{C([0,T];\mathcal{H})}^2=M_{5}^2\|g\|_{C([0,T];\mathcal{H})}^2,$$
which gives
\begin{equation}\label{cndp 2a}
    \|w^f\|_{C([0,T];\mathcal{H})}\leq M_{5} \|g\|_{C([0,T];\mathcal{H})},
\end{equation}
where $M_5=\frac{1}{\inf_{\xi\in\mathcal{I}}\mu_\xi m_{\tilde{\sigma}}}.$ Using that $g(t)=\tilde{f}(t)-f(t)-(\tilde{\sigma}(t)-\sigma(t))\mathcal{M}v(t;\sigma)$  and taking into account the properties of the norm, we get from \eqref{cndp 2a}
\begin{equation}\begin{split}\label{cndp 3a}
    \|w^f\|_{C([0,T];\mathcal{H})}&\leq M_{5} \|g\|_{C([0,T];\mathcal{H})}\\
    &\leq M_{5}\big(\|\tilde{f}-f\|_{C([0,T];\mathcal{H})}+\|\tilde{\sigma}-\sigma\|_{C[0,T]}\|\mathcal{M}v\|_{C([0,T];\mathcal{H})}\big)\\
    &\leq M_{6}\big(\|\tilde{f}-f\|_{C([0,T];\mathcal{H})}+\|\tilde{\sigma}-\sigma\|_{C[0,T]}\|v\|_{C([0,T];\mathcal{H}^1)}\big).
  \end{split}  
\end{equation}
In view of $v(t;\sigma)=v^h(t;\sigma)+v^f(t;\sigma)$ and using \eqref{cndp ur} and \eqref{cndp ui}, we have
\begin{equation}\begin{split}\label{cndp ua}
\|v\|_{C([0,T];\mathcal{H}^1)}&\leq \|v^h\|_{C([0,T];\mathcal{H}^1)}+\|v^f\|_{C([0,T];\mathcal{H}^1)}\\
&\leq M_7\big(\|h\|_{\mathcal{H}^1}+\|f\|_{C([0,T];\mathcal{H}^\frac{1}{2})}\big).
\end{split}\end{equation}
Substituting \eqref{cndp ua} and \eqref{cndp a} into \eqref{cndp 3a}, we get
\begin{equation}\begin{split}\label{cndp 4a}
\|w^f\|_{C([0,T];\mathcal{H})}&\leq M_{8}\big(\|\tilde{h}-h\|_{\mathcal{H}^{2+\gamma}}+\|\tilde{f}-f\|_{C([0,T];\mathcal{H}^{\frac{3}{2}+\gamma})}+\|\tilde{E}-E\|_{X^\alpha[0,T]}\big).
\end{split}\end{equation}
Using Lemma \ref{q_a>a ui} and estimate \eqref{EST: Mittag 0}, we have
\begin{equation*}
    \begin{split}
        \|w^h(t;\tilde{\sigma})\|_{\mathcal{H}}^2&\leq \sum_{\xi\in\mathcal{I}}|\tilde{h}_\xi-h_\xi|^2|E_{\alpha,1}(-\mu_\xi m_{\tilde{\sigma}} t^\alpha)|^2\\
        &\leq \sum_{\xi\in\mathcal{I}}|\tilde{h}_\xi-h_\xi|^2=\|\tilde{h}-h\|_\mathcal{H}^2,
    \end{split}
\end{equation*}
which gives
\begin{equation}\label{cndp 5a}
    \|w^h\|_{C([0,T];\mathcal{H})}\leq \|\tilde{h}-h\|_\mathcal{H}.
\end{equation}
In view of $\tilde{v}(t;\tilde{\sigma})-v(t;\sigma)=w(t;\tilde{\sigma})=w^h(t;\tilde{\sigma})+w^f(t;\tilde{\sigma})$ the estimates \eqref{cndp 4a} and \eqref{cndp 5a} give us
\begin{align*}
    \|\tilde{v}-v\|_{C([0,T];\mathcal{H})}&\leq M_{9}\big(\|\tilde{h}-h\|_{\mathcal{H}^{2+\gamma}}+\|\tilde{f}-f\|_{C([0,T];\mathcal{H}^{\frac{3}{2}+\gamma})}+\|\tilde{E}-E\|_{X^\alpha[0,T]}\big).
\end{align*}
This completes the proof.
\end{proof}

\subsubsection{Uniqueness of $(\sigma,v)$}\label{S:uniqueness} In this subsection, we prove the uniqueness of the generalized solution $(\sigma,v).$

\begin{thm}\label{uniqueness (a,v)}
    Let Assumption \ref{A:INV} hold. Then there is unique generalized solution $(\sigma,v)$ of Problem \ref{P:DIP}.
\end{thm}
\begin{proof}
Let $(\sigma,v)$ and $(\tilde{\sigma},\tilde{v})$ be two different solutions of the Problem \ref{P:DIP}, that is,
\begin{equation}\label{EQ:(a,v)}
  \begin{cases}
  &\mathcal{D}_t^\alpha v(t;\sigma)+\sigma(t)\mathcal{M}v(t;\sigma)=f(t),\\
  &v(0;\sigma)=h,\\
  &F[v(t;\sigma)]=E(t),\\
  \end{cases}
\end{equation}
and
\begin{equation}\label{EQ:(b,v)}
\begin{cases}
  &\mathcal{D}_t^\alpha \tilde{v}(t;\tilde{\sigma})+\tilde{\sigma}(t)\mathcal{M}\tilde{v}(t;\tilde{\sigma})=f(t),\\
  &\tilde{v}(0;\tilde{\sigma})=h,\\
  &F[\tilde{v}(t;\tilde{\sigma})]=E(t),
  \end{cases}
\end{equation}
respectively. Then in view of the estimates \eqref{a-b} and \eqref{u-v}, we obtain
\begin{equation}\label{a-b zero}
    \|\tilde{\sigma}-\sigma\|_{C[0,T]}\leq 0,
\end{equation}
\begin{equation}\label{u-v zero}
  \|\tilde{v}-v\|_{C([0,T];\mathcal{H})}\leq 0.  
\end{equation}
The estimates \eqref{a-b zero} and \eqref{u-v zero} imply $\sigma=\tilde{\sigma}$ and $v=\tilde{v},$ respectively. The uniqueness of the generalized solution $(\sigma,v)$ is proved.
\end{proof}

\begin{proof}[Proof of Theorem \ref{Th:main theorem}]In view of Theorem \ref{thm:MainThmDP}, Theorems \ref{existence (a,v)}, \ref{uniqueness (a,v)} allow us to conclude the existence and uniqueness of solution $(\sigma,v)$ of Problem \ref{P:DIP}. This with Theorem \ref{continuous} gives us a well-posedness of Problem \ref{P:DIP}. Thus, completes the proof of Theorem \ref{Th:main theorem}.
\end{proof}

\section{Examples of operator $\mathcal{M}$}\label{S:L}

In this section, we give several examples of the settings in which our direct and inverse problems are applicable. Of course, there are many other examples; here we collect those for which different types of partial differential equation are of particular importance.

\begin{itemize}
  \item {\bf Sturm-Liouville problem.}\\
  First, we describe the setting of the Sturm-Liouville operator. Let $\mathcal{M}$ be the ordinary second-order differential operator in $L^2(a,b)$ generated by the differential expression
\begin{equation}\label{SL} \mathcal{M}u=-u''(x),\,\,a<x<b,\end{equation} and one of the boundary conditions \begin{equation}\label{SL_B} a_1u'(b)+b_1u(b)=0,\,a_2u'(a)+b_2u(a)=0,\end{equation} or \begin{equation}\label{SL_B1} u(a)=\pm u(b),\,u'(a)=\pm u'(b),\end{equation} where $a^2_1+a^2_2>0,\,b_1^2+b_2^2>0$ and $\alpha_j,\, \beta_j,\,j=1,2,$ are some real numbers.

It is known (\cite{Naimark}) that the Sturm-Liouville problem for equation \eqref{SL} with boundary conditions \eqref{SL_B} or with boundary conditions \eqref{SL_B1} is self-adjoint in $L^2(a,b).$ It is also known that the self-adjoint problem has real eigenvalues and their eigenfunctions form a complete orthonormal basis in $L^2(a,b).$
\end{itemize}

\begin{itemize}
  \item {\bf Differential operator with involution.}\\
As a next example, we consider the differential operator with involution in $L^2(0,\pi)$ generated by the expression
\begin{equation}\label{DOI} \mathcal{M}u=u''(x)-\varepsilon u''(\pi-x),\,\,0<x<\pi,\end{equation} and homogeneous Dirichlet conditions \begin{equation}\label{DOI_B} u(0)=0,\,u(\pi)=0,\end{equation} where $|\varepsilon|<1$ is some real number.

The non-local functional-differential operator \eqref{DOI}-\eqref{DOI_B} is self-adjoint \cite{torebek}. For $|\varepsilon|<1,$ the operator \eqref{DOI}-\eqref{DOI_B} has the following eigenvalues
\begin{align*}\mu_{2k}=4(1+\varepsilon)k^2,\,k\in \mathbb{N},\,\,\, \textrm{and} \,\,\, \mu_{2k+1}=(1-\varepsilon)(2k+1)^2,\,k\in \mathbb{N}\cup\{0\},\end{align*}
and corresponding eigenfunctions
\begin{align*}&u_{2k}(x)=\sqrt{\frac{2}{\pi}}\sin{2kx},\,k\in \mathbb{N},\\& u_{2k+1}(x)=\sqrt{\frac{2}{\pi}}\sin{(2k+1)x},\,k\in \mathbb{N}\cup \{0\}.\end{align*}
\end{itemize}

\begin{itemize}
  \item {\bf Fractional Sturm-Liouville operator.}\\
We consider the operator generated by the integro-differential expression
\begin{equation}\label{FSL} \mathcal{M}u=\mathcal{D}_{a+}^\alpha D_{b-}^{\alpha}u,\,a<x<b,\end{equation}
and the conditions \begin{equation}\label{FSL_B}I_{b-}^{1-\alpha}u(a)=0,\,I_{b-}^{1-\alpha}u(b)=0,\end{equation} where
$\mathcal{D}_{a+}^\alpha$ is the left Caputo derivative of order $\alpha  \in
\left( {1/2,1} \right],$  $D_{b-}^\alpha$ is the right Riemann-Liouville derivative of order $\alpha \in
\left( {1/2,1} \right]$ and $I_{b-}^\alpha $ is the right Riemann-Liouville integral of order $\alpha \in
\left( {1/2,1} \right]$ (see \cite{KST06}).
The fractional Sturm-Liouville operator \eqref{FSL}-\eqref{FSL_B} is self-adjoint and positive in $L^2 (a, b)$ (see \cite{TT16}). The spectrum of the fractional Sturm-Liouville operator \eqref{FSL}-\eqref{FSL_B} is
discrete, positive, and real-valued, and the system of eigenfunctions is a complete orthogonal basis in $L^2 (a, b).$ For more properties of the operator generated by the problem \eqref{FSL}-\eqref{FSL_B} we refer to \cite{TT18, TT19}.
\end{itemize}

\begin{itemize}
    \item {\bf Second order elliptic operator $\mathcal{M}.$}\\
    Let $\Omega$ be an open bounded domain in $\mathbb{R}^d\, (d\geq 1)$ with a smooth boundary (for example, of $C^\infty$ class).

Let $L^2(\Omega)$ be the usual $L^2$-space with the inner product $(\cdot, \cdot)$ and 
let $\mathcal{M}$ be the elliptic operator defined for $g\in \mathcal{D}(\mathcal{M}):=H^2(\Omega)\cap H^1_0(\Omega)$ as
$$
\mathcal{M}g(x)=-\sum_{i,j=1}^{d}\partial_j(a_{ij}(x)\partial_j g(x))+c(x)g(x),\;x\in \Omega,
$$
with Dirichlet boundary condition
$$g(x)=0,\;x\in \partial \Omega,$$
where $a_{ij}=a_{ji}\,(1\leq i,j\leq d)$ and $c\geq 0$ in $\overline{\Omega}$. Moreover, assume that $a_{ij}\in C^1(\overline{\Omega}),\,c\in C(\overline{\Omega}),$ and there exists a constant $\delta>0$ such that 
$$
\delta\sum_{i=1}^d \xi_i^2\leq \delta\sum_{i=1}^d a_{ij}(x)\xi_i\xi_j, \; \forall x\in\overline{\Omega}, \,\forall (\xi_1,...,\xi_d)\in\mathbb{R}^d. 
$$
Then the elliptic operator $\mathcal{M}$ has the eigensystem $\{\mu_n,\omega_n\}_{n=1}^\infty$ such that $0<\mu_1\leq\mu_2\leq \cdots ,\mu_n\rightarrow\infty$ as $n\rightarrow\infty$ and $\{\omega_n\}_{n=1}^\infty$ forms an orthonormal basis of $L^2(\Omega)$ (\cite{SY11}).
\end{itemize}

\begin{itemize}
  \item {\bf Harmonic oscillator.}\\
For any dimension $d\geq1$, let us consider the harmonic oscillator,
$$
\mathcal{M}:=-\Delta+|x|^{2}, \,\,\, x\in\mathbb R^{d}.
$$
The operator $\mathcal{M}$ is an essentially self-adjoint operator on $C_{0}^{\infty}(\mathbb R^{d})$. It has a discrete spectrum, consisting of the eigenvalues
$$
\mu_{k}=\sum_{j=1}^{d}(2k_{j}+1), \,\,\, k=(k_{1}, \cdots, k_{d})\in\mathbb N^{d},
$$
with the corresponding eigenfunctions
$$
\varphi_{k}(x)=\prod_{j=1}^{d}P_{k_{j}}(x_{j}){\rm e}^{-\frac{|x|^{2}}{2}},
$$
which form an orthogonal basis in $L^{2}(\mathbb R^{d})$. Here $P_{l}(\cdot)$ is the $l$--th order Hermite polynomial,
$$
P_{l}(\xi)=a_{l}{\rm e}^{\frac{|\xi|^{2}}{2}}\left(x-\frac{d}{d\xi}\right)^{l}{\rm e}^{-\frac{|\xi|^{2}}{2}},
$$
where $\xi\in\mathbb R$, and
$$
a_{l}=2^{-l/2}(l!)^{-1/2}\pi^{-1/4}.
$$

\end{itemize}

\begin{itemize}
  \item {\bf Anharmonic oscillator.}\\
Another class of examples are anharmonic oscillators (see for instance
\cite{HR82}), which are operators on $L^2(\mathbb{R})$ of the form
$$
\mathcal{M}:=-\frac{d^{2k}}{dx^{2k}} +x^{2l}+p(x), \,\,\, x\in\mathbb R,
$$
for integers $k,l\geq 1$ and with $p(x)$ being a polynomial of degree $\leq 2l-1$ with real coefficients. More general case on $\mathbb{R}^n$ where a prototype operator is of the form
$$
\mathcal{M}:=-(\Delta )^k +|x|^{2l},
$$
where $k, l$ are integers $\geq 1$ see \cite{ChDR21}. 
\end{itemize}

\begin{itemize}
  \item {\bf Landau Hamiltonian in 2D.}\\
The next example is one of the simplest and most interesting models of Quantum Mechanics, that is, the Landau Hamiltonian.

The Landau Hamiltonian in 2D is given by
\begin{equation*} \label{eq:LandauHamiltonian}
\mathcal{M}:=\frac{1}{2}\left(\left(i\frac{\partial}{\partial x}-B
y\right)^{2}+\left(i\frac{\partial}{\partial y}+B x\right)^{2}\right),
\end{equation*}
acting on the Hilbert space $L^{2}(\mathbb R^{2})$, where $B>0$ is some constant. The spectrum of
$\mathcal{M}$ consists of infinite number of eigenvalues (see \cite{F28, L30}) with infinite multiplicity, of the form
\begin{equation*} \label{eq:HamiltonianEigenvalues}
\mu_{n}=(2n+1)B, \,\,\, n=0, 1, 2, \dots \,,
\end{equation*}
and the corresponding system of eigenfunctions (see \cite{ABGM15, HH13})
{\small
\begin{equation*}
\label{eq:HamiltonianBasis} \left\{
\begin{split}
e^{1}_{k, n}(x,y)&=\sqrt{\frac{n!}{(n-k)!}}B^{\frac{k+1}{2}}\exp\Big(-\frac{B(x^{2}+y^{2})}{2}\Big)(x+iy)^{k}L_{n}^{(k)}(B(x^{2}+y^{2})), \,\,\, 0\leq k, {}\\
e^{2}_{j, n}(x,y)&=\sqrt{\frac{j!}{(j+n)!}}B^{\frac{n-1}{2}}\exp\Big(-\frac{B(x^{2}+y^{2})}{2}\Big)(x-iy)^{n}L_{j}^{(n)}(B(x^{2}+y^{2})), \,\,\, 0\leq j,
\end{split}
\right.
\end{equation*}}
where $L_{n}^{(\alpha)}$ are the Laguerre polynomials given by
$$
L^{(\alpha)}_{n}(t)=\sum_{k=0}^{n}(-1)^{k}C_{n+\alpha}^{n-k}\frac{t^{k}}{k!}, \,\,\, \alpha>-1.
$$

\end{itemize}

\begin{itemize}
  \item {\bf The restricted fractional Laplacian.}\\
  On the other hand, one can define a fractional Laplacian operator by using the integral representation in terms of hypersingular kernels,
  $$\left(-\Delta_{\mathbb{R}^n}\right)^sg(x)=C_{d,s}\, \textrm{P.V.} \int\limits_{\mathbb{R}^n}\frac{g(x)-g(\xi)}{|x-\xi|^{n+2s}}d\xi,$$ where $s\in (0,1).$

  In this case, we realize the zero Dirichlet condition by restricting the operator to act only on functions that are zero outside of the bounded domain $\Omega\subset\mathbb{R}^n.$ Caffarelli and Siro \cite{CS17} called the operator defined in such a way as the restricted fractional Laplacian $\left(-\Delta_{\Omega}\right)^s.$ Such, $\left(-\Delta_{\Omega}\right)^s$ is a self-adjoint operator in $L^2(\Omega),$ with a discrete spectrum $\mu_{s,k}>0,\,\,k\in \mathbb{N}.$ The corresponding set of eigenfunctions $\{V_{s,k}(x)\}_{k\in \mathbb{N}},$ normalized in $L^2(\Omega),$ gives an orthonormal basis.
\end{itemize}

\section{Examples of functional $F$}\label{S:F}

In this section, as an illustration, we give several examples of the functional $F.$ Of course, there are many other examples, but here we only collect some of them.
\begin{itemize}
    \item {\bf The measurement is the total energy output from the body.}\\
    $$F[v(t,\cdot)]:=\int_\Omega v(x,t)dx,\;t\in[0,T],$$
    where $\Omega$ is a bounded subset of $\mathbb{R}^d (d\geq 1).$
\end{itemize}
\begin{itemize}
    \item {\bf Measurement at an internal point.}\\
    $$F[v(t,\cdot)]:=v(t,x^*),\;t\in[0,T],\;x^*\in \Omega,$$
    where $\Omega$ is an open bounded subset of $\mathbb{R}^d (d\geq 1).$
\end{itemize}
\begin{itemize}
    \item {\bf The measurement is the normal derivative of $v$ at one of the boundary points.}\\
    $$F[v(t,\cdot)]:=\frac{\partial v}{\partial \vec{n}}(t,x^*),\;t\in[0,T],\;x^*\in \partial\Omega,$$
    where $\Omega$ is open bounded subset of $\mathbb{R}^d (d\geq 1).$
\end{itemize}

\section{Value of $\gamma$ in particular cases of $\mathcal{M}$ and ${F}$}\label{S:gamma}

In this section, for the particular cases of $\mathcal{M}$ and ${F}$ we show how to find the value of $\gamma$ in \eqref{gamma F}.

Let $\mathcal{H}$ be $L^2(0,1)$ and let $\mathcal{M}$ be particular case of \textbf{Sturm-Liouville problem}, for example, $\mathcal{M}v=-v_{xx},\; x\in (0,1),$ with homogeneous Dirichlet boundary condition. Then the operator has the eigensystem $\{k^2, \sqrt{2}\sin{k\pi x}\}_{k\in\mathbb{N}}.$ 
\begin{itemize}
    \item Let $F$ be \textbf{the measurement is the total energy output from the body}, that is, $F[v(t,\cdot)]:=\int_{0}^1 v(t,x)dx.$ 

Since 
\begin{equation*}F[\omega_k]=\int_0^1 \sqrt{2}\sin{k\pi x}dx=\frac{1+(-1)^{k+1}}{k\pi}=\begin{cases}
    &\frac{2\sqrt{2}}{k\pi},\;\text{if},\;k=2n-1\;(n\in\mathbb{N});\\
    &0, \;\text{if},\;k=2n\;(n\in\mathbb{N}),
\end{cases}\end{equation*} 
we have $$\sum_{k\in\mathbb{N}} |F[\omega_k]|^2<\infty.$$ From this we see that $\gamma$ in \eqref{gamma F} can be taken to be $0.$
\item Let $F$ be \textbf{the measurement at an internal point}, for example, $F[v(t,\cdot)]=v(t,\frac{1}{2}).$
Then we see that
$$
F[\omega_k]=\sqrt{2}\sin{\frac{k\pi}{2}}=\begin{cases}
     &-\sqrt{2},\;\; \text{if},\;  k=4n-1,\;(n\in \mathbb{N});  \\
     & 0,\;\; \text{if},\;  k=2n,\;\;(n\in \mathbb{N});  \\
     & \sqrt{2},\;\; \text{if},\;  k=4n-3,\;\;(n\in \mathbb{N}).
\end{cases}
$$
From this we have
$$\sum_{k\in\mathbb{N}} \frac{|F[\omega_\xi]|^2}{\mu_k}<\infty.$$ From this we see that $\gamma$ in \eqref{gamma F} can be taken to be $\frac{1}{2}.$
\item Let $F$ be \textbf{the measurement is the normal derivative of $v$ at one of the boundary points}, for example, $F[v(t,\cdot)]=u_x(t,1).$ Since
$$F[\omega_k]=\sqrt{2}k\pi \cos{k\pi}=\begin{cases}
    &-\sqrt{2}k\pi,\;\;\text{if},\;k=2n-1,\;(n\in \mathbb{N});  \\
    &\sqrt{2}k\pi,\;\;\text{if},\;k=2n,\;\;(n\in \mathbb{N}),
\end{cases}$$
we have $$\sum_{k\in\mathbb{N}} \left|\frac{F[\omega_k]}{\mu_k}\right|^2<\infty.$$ From this we see that $\gamma$ in \eqref{gamma F} can be taken to be $1.$
\end{itemize}

\section{Appendix A}

In this section, we record several classical theorems from functional analysis used in this paper.

\begin{thm}\label{Th:BF}\cite[Theorem 1.9]{KST06}[Banach fixed point theorem] Let $X$ be a Banach space and let $A:X\rightarrow X$ be the map such that
$$\|Au-Av\|\leq\beta \|u-v\|\;\;(0 <\beta <1)$$
holds for all $u,\,v\in X.$ Then the operator $A$ has a unique fixed point $u^*\in X$ that is, $Au^*=u^*.$    
\end{thm}

\begin{thm}\label{Th:AA}\cite[Theorem 1.8]{KST06}[Arzel\'a-Ascoli theorem]
A necessary and sufficient condition that a subset of continuous functions $U,$ which are defined on the closed interval $[a,b],$ be relatively compact in $C[a,b]$ is that this subset be uniformly bounded and equicontinuous.
\end{thm}

In this statement,
\begin{itemize}
    \item $U\subset C[a,b]$ is uniformly bounded means that there exists a number $C$ such that $$|\varphi(x)|\leq C$$ for all $x\in[a,b]$ and for all $\varphi\in U,$ and
    \item $U\subset C[a,b]$ is equicontinuous means that: for every $\varepsilon>0$ there is a $\delta=\delta(\varepsilon)>0$ such that $$|\varphi(x_1)-\varphi(x_2)|<\varepsilon$$
  holds for all $x_1,\,x_2\in [a,b]$ such that $|x_1-x_2|<\delta$ and for all $\varphi\in U.$
\end{itemize}

\begin{thm}\label{Th:ShF}\cite[Theorem 1.7]{KST06}[Shauder's fixed point theorem]
    Let $U$ be a closed convex subset of $C[a,b],$ and let $A:U\rightarrow U$ be the map such that the set $\{Au:\;u\in U\}$ is relatively compact in $C[a,b].$ 
    Then the operator $A$ has at least one fixed point $u^*\in U$ i.e. $Au^*=u^*.$
\end{thm}
\begin{thm}\label{LFB}\cite[p. 77]{KF} Let $H$ be Hilbert space.
  Then linear functional $F:H\rightarrow \mathbb{R}$ is continuous if and only if it is bounded on $H.$  
\end{thm}

\end{document}